\documentclass[11pt]{amsart}
\usepackage{amssymb}
\usepackage{amsfonts}
\usepackage{amsmath}
\usepackage{graphicx}
\usepackage{xcolor}
\usepackage{mathrsfs}
\usepackage{stmaryrd}
\usepackage{epsfig,color}
\usepackage{blindtext}
\usepackage{enumerate}
\usepackage{hyperref}
\usepackage{url}
\usepackage{bbm}
\usepackage{filecontents}
\usepackage{nicefrac,mathtools}
\usepackage{bm}  
\usepackage[nocompress]{cite}
\DeclareGraphicsExtensions{.pdf,.jpeg,.png}
\usepackage{epstopdf}
\usepackage{cancel} 
\usepackage[normalem]{ulem} 
\usepackage{verbatim} 
\usepackage{enumitem} 
\usepackage{tikz-cd}
\usetikzlibrary{cd}
\usepackage{color}
\usepackage[msc-links, lite]{amsrefs}
\usepackage{geometry}
\usepackage{tikz}
\usetikzlibrary{decorations.markings}
\usetikzlibrary{arrows.meta}

\usepackage{extarrows}
\newtheorem{theorem}{Theorem}[section]

\newtheorem{proposition}[theorem]{Proposition}
\newtheorem{lemma}[theorem]{Lemma}
\newtheorem{corollary}[theorem]{Corollary}

\theoremstyle{definition}
\newtheorem{definition}[theorem]{Definition}
\newtheorem{example}[theorem]{Example}

\numberwithin{equation}{section}

\numberwithin{equation}{section}

\makeatletter
\let\save@mathaccent\mathaccent
\newcommand*\if@single[3]{%
\setbox0\hbox{${\mathaccent"0362{#1}}^H$}%
\setbox2\hbox{${\mathaccent"0362{\kern0pt#1}}^H$}%
\ifdim\ht0=\ht2 #3\else #2\fi
}

\makeatother

\makeatletter
\newcommand*{\transpose}{%
{\mathpalette\@transpose{}}%
}
\newcommand*{\@transpose}[2]{%
\raisebox{\depth}{$\m@th#1\intercal$}%
}
\makeatother
\usetikzlibrary{hobby}

 \usetikzlibrary{decorations}
\makeatletter
\def\pgfutil@Repeat#1#2{#2\ifnum#1>0
  \expandafter\pgfutil@firstofone\else\expandafter\pgfutil@gobble\fi
  {\expandafter\pgfutil@Repeat\expandafter{\the\numexpr#1-1\relax}{#2}}}
\tikzset{
  dash between/.code args={#1 and #2}{%
    \tikz@addoption{%
      \pgfgetpath\currentpath
      \pgfprocessround{\currentpath}{\currentpath}%
      \pgf@decorate@parsesoftpath{\currentpath}{\currentpath}%
      \pgfmathsetlengthmacro\firstpart{(#1)*\pgf@decorate@totalpathlength}%
      \pgfmathsetlengthmacro\secondpart{(#2-(#1))*\pgf@decorate@totalpathlength}%
      \pgfmathsetlengthmacro\thirdpart{(1-(#2))*\pgf@decorate@totalpathlength}%
      \edef\thirdpart{{\thirdpart}{0pt}}%
      \edef\firstpart{{\firstpart}{0pt}}%
      \pgfmathsetlengthmacro\secondpartlength{\pgfkeysvalueof{/tikz/dash between on}
                                            +(\pgfkeysvalueof{/tikz/dash between off})}%
      \pgfmathtruncatemacro\repetitions{\secondpart/\secondpartlength}%
      \pgfmathsetlengthmacro\secondexpand{\secondpart/\repetitions-\secondpartlength}%
      \edef\secondexpand{\the\dimexpr\pgfkeysvalueof{/tikz/dash between off}+\secondexpand\relax}%
      \edef\secondpart{%
        \pgfutil@Repeat{\the\numexpr\repetitions-1\relax}%
          {{\pgfkeysvalueof{/tikz/dash between on}}{\secondexpand}}%
      }%
      \edef\tikz@temp{\firstpart\secondpart\thirdpart}%
      \expandafter\pgfsetdash\expandafter{\tikz@temp}{+0pt}%
    }
  }
}
\makeatother
\tikzset{
  dash between style/.is choice,
  dash between style/dotted/.style        ={dash between on=\pgflinewidth,dash between off=2pt},
  dash between style/densely dotted/.style={dash between on=\pgflinewidth,dash between off=1pt},
  dash between style/loosely dotted/.style={dash between on=\pgflinewidth,dash between off=4pt},
  dash between style/dashed/.style        ={dash between on=3pt,dash between off=2pt},
  dash between style/loosely dashed/.style={dash between on=3pt,dash between off=6pt},
  dash between style/densely dashed/.style={dash between on=3pt,dash between off=2pt},
  dash between style/no/.style={dash between on=0pt, dash between off=1pt},
  dash between on/.initial=\pgflinewidth,
  dash between off/.initial=2pt,
  middle dotted line/.style={
    thick,
    dash between=.35 and .65}}

\newcommand\QQ{\mathbb{Q}}
\newcommand\CC{\mathbb{C}}

\newcommand\ZZ{\mathbb{Z}}

\newcommand\OO{\mathcal{O}}

\DeclareMathOperator{\pro}{pro}

\DeclareMathOperator{\homo}{Hom}
\DeclareMathOperator{\Tor}{Tor}
\let\Im\relax
\DeclareMathOperator{\Im}{Im}
\DeclareMathOperator{\Ker}{Ker}

\let\Im\relax
\DeclareMathOperator{\Im}{Im}

\DeclareMathOperator{\Spec}{Spec}
\DeclareMathOperator{\Lie}{Lie}
\DeclareMathOperator{\GL}{GL}
\DeclareMathOperator{\End}{End}
\DeclareMathOperator{\cont}{cont}

\usepackage[all]{xy}

\hypersetup{pdfborder=0 0 0}

\title{Continuous Mal'cev $\QQ_p$-Completion of Pro-$p$ Groups}
\author{Runjie Hu, Guozhen Wang}
\newcommand{\Addresses}{{
  \bigskip
  \footnotesize

  Runjie Hu, \textsc{Department of Mathematics, Texas A\&M University,
    College Sta, TX 77843, US}\par\nopagebreak
  \textit{E-mail address}, \texttt{ runjie.hu@tamu.edu}

  Guozhen Wang, \textsc{Fudan University Jiangwan Campus,
2005 Songhu Road, Shanghai, China}\par\nopagebreak
  \textit{E-mail address}, \texttt{ wangguozhen@fudan.edu.cn}
}}
\date{}
\begin{document}

\maketitle

\begin{abstract}
We give three explicit constructions of the continuous Mal'cev $\QQ_p$-completion of a topologically finitely generated pro-$p$ group, using Tannakian formalism, Hopf algebras and $p$-adic analytic groups. We study properties of the continuous Mal'cev $\QQ_p$-completion via these explicit constructions.
\end{abstract}

\section*{Introduction}

Mal'cev completion is a fundamental method that replaces a discrete group by a rational unipotent algebraic group and its nilpotent Lie algebra. The origin of the theory is due to Mal'cev (\cite{Malcev-nilpotent-torsion-free-groups}) on finitely generated torsion-free nilpotent groups. A second formulation was introduced by Quillen in \cite{Quillen-rational-homotopy}.

Mal'cev completion has also been extended to a continuous $p$-adic version for pro-$p$ groups. Hain and Matsumoto introduced the definition of the continuous Mal'cev $\QQ_p$-completion of a profinite group in \cite{Hain-Matsumoto-malcev-completion}*{\S A.2}. Pridham in \cite{Pridham-malcev-completion}*{\S 1.1} defined a relative version of continuous Mal'cev completion which generalizes the usual continuous Mal'cev $\QQ_p$-completion. A more general formulation for the continuous Mal'cev completion of a topological group with respect to a topological field was later used by Betts in $p$-adic anabelian geometry (\cite{Betts-thesis}).

We briefly recall the definition of the continuous Mal'cev $\QQ_p$-completion. Let $G$ be a pro-$p$ group, which is finitely generated topologically. The continuous Mal'cev $\QQ_p$-completion of $G$ consists of a pro-unipotent group $G\widehat{\otimes} \QQ_p$ over $\QQ_p$, together with a continuous homomorphism $G\rightarrow (G\widehat{\otimes} \QQ_p)(\QQ_p)$, such that, for any continuous pro-unipotent group $U$ over $\QQ_p$, every continuous homomorphism $f:G\rightarrow U(\QQ_p)$ uniquely factors through $(G\widehat{\otimes} \QQ_p)(\QQ_p)$. 

Rather than introducing a new definition, the purpose of this paper is to construct the continuous Mal'cev $\QQ_p$-completion in three frameworks, as follows.

\begin{itemize}[leftmargin=0.25in]
    \item The first construction uses Tannakian formalism. The category of continuous unipotent $\QQ_p$-representations of $G$ yields a canonical pro-unipotent group over $\QQ_p$.
    \item The second construction is a continuous $p$-adic analogue of Quillen's complete Hopf algebra construction in \cite{Quillen-rational-homotopy}*{Appendix A}. Let $\widehat{\ZZ}_p[[G]]:=\varprojlim_U \widehat{\ZZ}_p[G/U]$, where $\{U\}$ ranges over the open normal subgroups of $G$. This completed group algebra is the $p$-adic analogue of the ordinary group algebra of $G$. Let $I$ be the augmentation ideal of $\widehat{\ZZ}_p[[G]]$. We then introduce the ``complete Hopf algebra'' $\QQ_p[[G]]^{\wedge}_I=\varprojlim_n \QQ_p\otimes_{\widehat{\ZZ}_p} (\widehat{\ZZ}_p[[G]]/\overline{I^n})$. The second construction is obtained by taking the group-like elements of this Hopf algebra.
    \item The third construction is a direct $\QQ_p$-analogue of Mal'cev's original construction for finitely generated nilpotent groups. Let $\Gamma^*$ denote the lower central series of $G$. The torsion-free quotient of each $G/\Gamma^r$ is analytically identified with an affine $p$-adic analytic manifold $\widehat{\ZZ}^{s_r}_p$. Under this identification, the group multiplication is polynomial and therefore extends from $\widehat{\ZZ}^{s_r}_p$ to $\QQ_p^{s_r}$. Finally, we take the inverse limit of these constructions for all $r$.
\end{itemize}

Our main result is the following theorem.

\begin{theorem}[Theorems \ref{Thm: equivalence between Malcev completion and Tannakian formalism of unipotent representations}, \ref{Thm: group-like elements of the Iwasawa algebra is the Malcev completion}, and \ref{Thm: equivalence between malcev completion and formal group law}]
Let $G$ be a topologically finitely generated pro-$p$ group. The three constructions described above yield pro-unipotent groups over $\QQ_p$ which are canonically isomorphic to the continuous Mal'cev $\QQ_p$-completion of $G$.
\end{theorem}

We also study several properties of the continuous Mal'cev $\QQ_p$-completion (Propositions \ref{Prop: Malcev completion is the inverse limit of Malcev completion of lower central series}, \ref{Prop: base change of Malcev completion} and Lemmas \ref{Lem: malcev completion of a finitely generated group is finitely generated}, \ref{Lem: Malcev completion preserves exactness of nilpotent groups}). Applications of these explicit constructions to topology and algebraic geometry will appear in a subsequent paper. 

\medskip

\subsubsection*{Acknowledgements}
R.H. is supported by 
NSF Grant 2247322. G.W. is partially supported by grants NSFC12325102, NSFC-12226002, the New Cornerstone Science Foundation, and Shanghai
Pilot Program for Basic Research–Fudan University 21TQ1400100 (21TQ002). We are grateful to Alexander Betts for generously sharing many of his ideas and arguments with us, particularly those used in Theorem \ref{Thm: equivalence between Malcev completion and Tannakian formalism of unipotent representations} and Theorem \ref{Thm: equivalence between mal'cev completion and group-like elements of Hopf algebra}.

\tableofcontents

\section{Continuous Mal'cev Completion}

In this section, we recall from \cite{Betts-thesis} the continuous Mal'cev completion of a topological group with respect to a topological field. We then establish several basic properties of this construction. Throughout this section, $k$ denotes a Hausdorff topological field of characteristic zero. 

\begin{definition}[\cite{Betts-thesis}*{Definition-Lemma 2.3.1}]\label{Fact: definition of continuous Malcev completion}
Let $G$ be a topological group. Consider the functor from pro-unipotent groups over $k$ to sets defined by $F(U)=\homo_{\cont}(G,U(k))$, which consists of continuous homomorphisms $G\rightarrow U(k)$. This functor is represented by a pro-unipotent group $G\widehat{\otimes} k$, which is called the \textbf{continuous Mal'cev $k$-completion} of $G$.
\end{definition}

Let $\Gamma^1G=G$ and $\Gamma^nG=(G,\Gamma^{n-1}G)$ be the lower central series of $G$. For each $n$, $\overline{\Gamma^n G}$ denotes the closure of $\Gamma^n G$ in $G$.

\begin{proposition}\label{Prop: Malcev completion is the inverse limit of Malcev completion of lower central series}
The natural morphism
$G\widehat{\otimes} k\rightarrow \varprojlim (G/\overline{\Gamma^n G})\widehat{\otimes} k$ is an isomorphism of pro-unipotent groups over $k$.
\end{proposition}

\begin{proof}
Let $U$ be a unipotent group over $k$. Then there exists $n$ such that $\Gamma U^n(k)=0$. Hence every continuous homomorphism $G\rightarrow U(k)$ kills $\Gamma^n G$ and therefore also its closure $\overline{\Gamma^n G}$, since $U(k)$ is Hausdorff. It follows that 
\[
\homo_{\cont}(G,U(k))\cong \varinjlim_n \homo_{\cont}(G/\overline{\Gamma^n G},U(k))\cong \varinjlim_n \homo((G/\overline{\Gamma^n G})\widehat{\otimes} k, U)\cong \homo(\varprojlim_n(G/\overline{\Gamma^n G})\widehat{\otimes} k, U)
\]
The universal property in Definition \ref{Fact: definition of continuous Malcev completion} therefore gives a canonical isomorphism $G\widehat{\otimes} k\cong \varprojlim (G/\overline{\Gamma^n G})\widehat{\otimes} k$.
\end{proof}

\begin{definition}
A topological group is \textbf{topologically finitely generated} if it contains a dense algebraically finitely generated subgroup.    
\end{definition}

\begin{definition}
Let $U$ be a pro-unipotent group over $k$. A subset $S$ of $U(k)$ is said to \textbf{generate $U$} if no proper closed subgroup scheme of $U$ contains $S$ in its set of $k$-points. The pro-unipotent group $U$ is called \textbf{finitely generated} if it admits a finite generating subset in this sense.
\end{definition}

\begin{lemma}\label{Lem: malcev completion of a finitely generated group is finitely generated}
If $G$ is a topologically finitely generated topological group, then its continuous Mal'cev completion $G\widehat{\otimes} k$ is finitely generated as a pro-unipotent group.
\end{lemma}

\begin{proof}
Let $U=G\widehat{\otimes} k$, and let $f:G\rightarrow U(k)$ be the universal homomorphism of Definition \ref{Fact: definition of continuous Malcev completion}. Choose a finite subset $S\subset G$ which algebraically generates a dense subgroup. Let $V\subset U$ be the smallest closed subgroup scheme whose $k$-points contain $f(S)$. By Proposition \ref{Fact: closedness of pro-unipotentness for algebraic group schemes}, $V$ is again pro-unipotent.

Write $U\cong \varprojlim_i U_i$, where each $U_i$ is a unipotent group over $k$. Let $V_i$ be the image of  $V\rightarrow U\rightarrow U_i$.

\textbf{Claim.} The canonical morphism $V\rightarrow \varprojlim_i V_i$ is an isomorphism.

\textit{Proof of the Claim.} 
Since $U\cong \varprojlim_i U_i$, its coordinate ring satisfies $\mathcal{O}(U)\varinjlim_i \mathcal{O}(U_i)$.
Let $I\subset \mathcal{O}(U)$ be the Hopf ideal defining $V$. Let $I_i$ be the kernel of $\mathcal{O}(U_i)\rightarrow \mathcal{O}(U)\rightarrow \mathcal{O}(U)/I$. Then $\mathcal{O}(U_i)/I_i$ is the coordinate ring of $V_i$. Moreover, $\varinjlim_i I_i=I$. Hence, $\OO(V)\cong \varinjlim_i \OO(V_i) $, which proves that $V\cong \varprojlim_i V_i$. This completes the proof of the claim.
\medskip

The functor of $k$-points preserves limits. The claim therefore identifies the subspace topology on $V(k)\subset U(k)$ with the inverse-limit topology on $\varprojlim_i V_i(k)$. Since each $V_i\subset U_i$ is a closed subgroup scheme, the subset $V_i(k)\subset U_i(k)$ is defined by polynomials, and is therefore closed by Proposition
\ref{Prop: topology of a unipotent group} and Lemma \ref{Lemma: continuity of kn}. Consequently,  $V(k)\subset U(k)$ is closed. Because $f(G)\subset V(k)$ and the algebraic subgroup generated by $S$ is dense in $G$, continuity of $f$ and closedness of $V(k)$ imply that  $f(G)\subset V(k)$. The universal property of $U=G\widehat{\otimes}k$ then forces the inclusion $V\hookrightarrow U$ to be an isomorphism.
\end{proof}

\begin{proposition}\label{Prop: base change of Malcev completion}
Let $G$ be a topologically finitely generated topological group. Let $k\subset K$ be a weakly cartesian pair of Hausdorff topological fields. Assume that $k$ is straight. Then there is a canonical isomorphism $G\widehat{\otimes}K \cong (G\widehat{\otimes} k)_K$, where the right-hand side denotes base change from $k$ to $K$.
\end{proposition}

\begin{proof}
Let $U$ be a unipotent group over $K$. The restriction of scalars from $K$ to $k$ allows us to regard $\Lie(U)$ as a possibly infinite-dimensional nilpotent $k$-Lie algebra. As a $k$-Lie algebra, $\Lie(U)$ is the filtered union of its finite-dimensional nilpotent $k$-Lie subalgebras $L_i$. By Proposition \ref{Prop: direct sum of weakly cartesian is weakly cartesian} and Proposition \ref{Prop: base change of vector spaces over weakly cartesian pairs}, each $L_i$ is closed in $\Lie(U)$ and is linearly homeomorphic to $k^{r_i}$ for some $r_i$. Let $U_i$ be the unipotent group over $k$ corresponding to the nilpotent Lie algebra $L_i$. Under the exponential map and the Baker-Campbell-Hausdorff formula, there is a canonical group isomorphism $U(K)\cong \varinjlim_i U_i(k)$.

Choose a finite subset $S\subset G$ which topologically generates $G$. Let $f:G\rightarrow U(K)$ be a continuous homomorphism, and consider the continuous map $\log\circ f:G\rightarrow \Lie(U)$. There exists a finite-dimensional nilpotent $k$-Lie subalgebra $L_i\subset \Lie(U)$ which contains $(\log\circ f)(S)$. Since $L_i\subset \Lie(U)$ is closed and the subgroup generated by $S$ is dense in $G$, $L_i$ contains $(\log\circ f)(G)$. Hence $f$ factors through the inclusion $U_i(k)\hookrightarrow U(K)$. Therefore, the canonical map $\varinjlim_i \homo_{\cont}(G,U_i(k))\rightarrow \homo_{\cont}(G,U(K))$ is bijective. 

By the universal property in Definition \ref{Fact: definition of continuous Malcev completion}, $\varinjlim_i \homo_{\cont}(G,U_i(k))\cong \varinjlim_i\homo (G\widehat{\otimes} k, U_i)$.  

Using the equivalence between pro-unipotent groups and pro-nilpotent Lie algebras in \cite{Betts-thesis}*{Theorem 2.1.4}, we obtain $\varinjlim_i\homo (G\widehat{\otimes} k, U_i)\cong \varinjlim_i \homo(\Lie(G\widehat{\otimes} k), \Lie(U_i))$. 

By Lemma \ref{Lem: malcev completion of a finitely generated group is finitely generated}, the pro-unipotent group $G\widehat{\otimes}k$, and hence its Lie algebra $\Lie(G\widehat{\otimes} k)$ is finitely generated. Choose a finite generating set $S'$ for $\Lie(G\widehat{\otimes} k)$. For every $k$-Lie algebra morphism $\phi:\Lie(G\widehat{\otimes} k)\rightarrow \Lie(U)$, the finite set $\phi(S')$ is contained in some finite-dimensional nilpotent $k$-Lie subalgebra $L_i$. Consequently, $\varinjlim_i \homo_k(\Lie(G\widehat{\otimes} k), \Lie(U_i))\cong \homo_k(\Lie(G\widehat{\otimes} k), \Lie(U))$.

By the tensor-restriction adjunction, $\homo_k(\Lie(G\widehat{\otimes} k), \Lie(U))\cong\homo_K(\Lie(G\widehat{\otimes} k)\otimes K, \Lie(U))$. Under the Lie-theoretic equivalence, the latter is naturally bijective to $\homo((G\widehat{\otimes} k)_K, U)$.

Thus, $(G\widehat{\otimes}k)_K$ represents the functor $U\rightarrow \homo_{\cont}(G,U(K))$. By Definition \ref{Fact: definition of continuous Malcev completion}, it is canonically isomorphic to $G\widehat{\otimes} K$.
\end{proof}

\begin{example}
Let $G$ be a topologically finitely generated abelian pro-$p$ group. Then $G$ is a finitely generated $\widehat{\ZZ}_p$-module. Its continuous Mal'cev $\QQ_p$-completion is the additive affine group scheme associated with the finite-dimensional $\QQ_p$-vector space $G\otimes_{\widehat{\ZZ}_p} \QQ_p$. \qed
\end{example}

\section{Tannakian Construction}

In this section, we construct continuous Mal'cev completion using the Tannakian formalism. We first review the Tannakian formalism and construct a pro-unipotent group $U_{G,k}$ for every topological group $G$ (Lemma \ref{Lem: Tannakian formalism of unipotent representations give us pro-unipotent groups}). We then show that the pro-unipotent group $U_{G,k}$ is the continuous Mal'cev completion of $G$ (Theorem \ref{Thm: equivalence between Malcev completion and Tannakian formalism of unipotent representations}).

\subsection{Construction by Tannakian Formalism}\,

For a field $F$, let $\mathbf{Vect}_F$ denote the category of finite-dimensional $F$-vector spaces. We refer to \cite{Deligne-tannakian-categories}*{Definition 1.7} for the definition of a \textbf{rigid tensor category}.

\begin{definition}[\cite{Deligne-tannakian-categories}*{Definition 2.19}]
Let $F$ be a field of characteristic zero. A small rigid abelian tensor category $(\mathcal{C},\otimes)$ satisfying $\End(\mathbf{1})=F$ is \textbf{neutral Tannakian over $F$} if it admits a faithful exact $F$-linear tensor functor $\omega:\mathcal{C}\rightarrow \mathbf{Vect}_F$.
\end{definition}

Let $(\mathcal{C},\otimes,\omega)$ be a neutral Tannakian category. Define a functor $\underline{\text{Aut}}^{\otimes}(\omega)$ from $F$-algebras to sets as follows. For an $F$-algebra $R$, let $\underline{\text{Aut}}^{\otimes}(\omega)(R)$ be the set of families $\{\lambda_X\}_{X\in \mathcal{C}}$, where each $\lambda_X$ is an $R$-linear automorphism of $\omega(X)\otimes_F R$ such that $\lambda_{X_1\otimes X_2}=\lambda_{X_1}\otimes \lambda_{X_2}$ for all $X_1,X_2\in \mathcal{C}$, $\lambda_{\mathbf{1}}=Id_R$, and, for every morphism $\alpha:X\rightarrow Y$, $\lambda_Y\circ (\omega(\alpha)\otimes \mathbf{1})=(\omega(\alpha)\otimes \mathbf{1})\circ \lambda_X:\omega(X)\otimes_F R\rightarrow \omega(Y)\otimes_F R$.

Recall the following theorem of the so-called \textbf{Tannakian formalism}.

\begin{theorem}[\cite{Deligne-tannakian-categories}*{Theorem 2.11}, Tannakian Formalism]\label{Fact: Tannakian Formalism}
Let $F$ be a field of characteristic zero. Let $(\mathcal{C},\otimes,\omega)$ be a neutral Tannakian category over $F$. Then
\begin{enumerate}[leftmargin=0.25in]
    \item the functor $\underline{\text{Aut}}^{\otimes}(\omega)$ is represented by an affine group scheme $A$ over $F$;
    \item the functor $\omega$ factors through a tensor categorical equivalence between $\mathcal{C}$ and the category $\mathbf{Rep}_{F}(A)$ of finite-dimensional $F$-representations of $A$.
\end{enumerate}
\end{theorem}

\begin{definition}
Let $F$ be a field of characteristic zero, and let $G$ be an abstract group. A finite-dimensional $F$-representation $\rho:G\rightarrow \GL(V)$ is \textbf{unipotent} if one of the following equivalent conditions holds:
\begin{enumerate}[leftmargin=0.25in]
    \item there exists a basis of $V$ with respect to which  $\rho(G)$ is contained in the group $\mathbb{U}_{m,F}(F)$ of upper-unitriangular matrices;
    \item there exists a $G$-stable flag $V=V_m\supset ...\supset V_1\supset 0$ such that $G$ acts trivially on each successive quotient $V_i/V_{i-1}$.
\end{enumerate}
\end{definition} 

Let $k$ be a Hausdorff topological field of characteristic zero, and let $V$ be an $m$-dimensional $k$-vector space. A choice of basis of $V$ identifies $\End(V)$ with $M_m(k)\cong k^{m^2}$, and therefore equips it with a topology. By Lemma \ref{Lemma: continuity of kn}, this topology is independent of the chosen basis.

\begin{definition}
Let $G$ be a topological group. A \textbf{continuous finite-dimensional $k$-representation} of $G$ is a continuous homomorphism $\rho:G\rightarrow \GL(V)$, where $V$ is a finite-dimensional $k$-vector space. A \textbf{morphism} $(V_1,\rho_1)\rightarrow (V_2,\rho_2)$ between continuous representations is a $G$-equivariant $k$-linear map $V_1\rightarrow V_2$. Let $\mathbf{Rep}_{k,\cont}(G)$ denote the category of continuous finite-dimensional $k$-representations of $G$, and let $\mathbf{URep}_{k,\cont}(G)$ denote its full subcategory of unipotent representations.
\end{definition}

\begin{proposition}
Each of $\mathbf{Rep}_{k,\cont}(G)$ and $\mathbf{URep}_{k,\cont}(G)$, equipped with its forgetful functor to $\mathbf{Vect}_k$, is a neutral Tannakian category over $k$.
\end{proposition}

\begin{proof}
Let $\mathbf{Rep}_{k}(G)$ denote the category of all finite-dimensional $k$-representations of $G$, without any continuity requirement. This is a rigid abelian tensor category whose unit object is the trivial representation $\mathbf{1}$, with $\End(\mathbf{1})=k$. The forgetful functor $\omega:\mathbf{Rep}_{k}(G)\rightarrow \mathbf{Vect}_k$ is faithful, exact, $k$-linear and compatible with tensor products. Thus, $(\mathbf{Rep}_{k}(G),\omega)$ is neutral Tannakian over $k$. It remains to verify that $\mathbf{URep}_{k,\cont}(G)$ and $\mathbf{Rep}_{k,\cont}(G)$ are full rigid abelian tensor subcategories of $\mathbf{Rep}_{k}(G)$. We prove only the case $\mathbf{Rep}_{k,\cont}(G)$ since the other case is analogous.

The inclusion $\mathbf{Rep}_{k,\cont}(G)\subset \mathbf{Rep}_{k}(G)$ is full. Kernels and cokernels of morphisms between continuous representations inherit continuous $G$-actions, so this subcategory is abelian. The trivial representation is continuous, so the unit object belongs to $\mathbf{Rep}_{k,\cont}(G)$. It remains to verify that $\mathbf{Rep}_{k,\cont}(G)$ is closed under tensor products and duals.

We first show that the tensor product of two continuous representations $(V,\rho_V)$ and $(W,\rho_W)$ is continuous. Set $m=\dim_k V$ and $n=\dim_k W$. After choosing bases, regard $\rho_V$ and $\rho_W$ as continuous maps $G\rightarrow k^{m^2}$ and $G\rightarrow k^{n^2}$. Each matrix coefficient of the tensor-product representation $\rho_{V\otimes W}:G\rightarrow k^{m^2n^2}$ is a product of a matrix coefficient of $\rho_V$ and a matrix coefficient of $\rho_W$. Hence, every matrix coefficient of $\rho_{V\otimes W}$ and therefore $\rho_{V\otimes W}$ are continuous. 

The dual representation is continuous because inversion and transpose are continuous polynomial operations on $\GL(V)$. This completes the proof.
\end{proof}

\begin{lemma}\label{Lemma: the construction of continuous Malcev completion by Tannakian formalism}
Let $U_{G,k}$ be the affine group scheme  associated to $(\mathbf{URep}_{k,\cont}(G),\omega)$. Then there exists a canonical homomorphism $f:G\rightarrow U_{G,k}(k)$ such that, for every continuous finite-dimensional unipotent $k$-representation $\rho_V:G\rightarrow \GL(V)$, the induced representation $\rho_{U_{G,k}}:U_{G,k}\rightarrow \GL(V)$ satisfies $\rho_V=\rho_{U_{G,k}}\circ f$.
\end{lemma}

\begin{proof}
By the construction in the proof of \cite{Deligne-tannakian-categories}*{Theorem 2.11}, the affine scheme $U_{G,k}$ represents $\underline{\text{Aut}}^{\otimes}(\omega)$. Each $g\in G$ defines a tensor automorphism of $\omega$ by acting on every continuous unipotent representation of $G$. Thus, $g$ determines an element of $\underline{\text{Aut}}^{\otimes}(\omega)(k)$. These tensor automorphisms define a canonical homomorphism $f:G\rightarrow U_{G,k}(k)$. Under the Tannakian equivalence (Theorem \ref{Fact: Tannakian Formalism}),
every $V\in \mathbf{URep}_{k,\cont}(G)$ becomes a representation of $U_{G,k}(k)=\underline{\text{Aut}}^{\otimes}(\omega)$ and its original $G$-action is recovered by restriction along $f$. Equivalently, $\rho_V=\rho_{U_{G,k}}\circ f$.
\end{proof}

\begin{lemma}\label{Lem: Tannakian formalism of unipotent representations give us pro-unipotent groups}
Retain the same notation as Lemma \ref{Lemma: the construction of continuous Malcev completion by Tannakian formalism}.
The affine group scheme $U_{G,k}$ is pro-unipotent.
\end{lemma}

\begin{proof}
Every nonzero object $V\in \mathbf{URep}_{k,\cont}(G)$ is unipotent and therefore has a nonzero $G$-fixed vector. Hence, $\homo_{\mathbf{URep}_{k,\cont}(G)}(\mathbf{1},V)\neq 0$. Under the Tannakian equivalence (Theorem \ref{Fact: Tannakian Formalism}), every finite-dimensional representation of $U_{G,k}$ has a nonzero invariant vector. Definition \ref{Definition: definition of unipotent groups} implies that $U_{G,k}$ is pro-unipotent.
\end{proof}

\subsection{Equivalence to Continuous Mal'cev Completion}\,

In this part, we prove that the pro-unipotent group $U_{G,k}$ constructed in Lemma \ref{Lem: Tannakian formalism of unipotent representations give us pro-unipotent groups} is the continuous Mal'cev completion of $G$.

\begin{lemma}
The canonical homomorphism $f:G\rightarrow U_{G,k}(k)$ of Lemma \ref{Lemma: the construction of continuous Malcev completion by Tannakian formalism} is continuous.
\end{lemma}

\begin{proof}
Write $U_{G,k}\cong \varprojlim_i U_i$, where each $U_i$ is a unipotent group over $k$. By the definition of the inverse-limit topology, it suffices to show that every composite $G\rightarrow U_i(k)$ is continuous. Choose a faithful finite-dimensional representation $U_i\hookrightarrow \GL(V_i)$. Composing with $U_{G,k}\rightarrow U_i$, we may regard $V_i$ as a representation of $U_{G,k}$. By the Tannakian equivalence (Theorem \ref{Fact: Tannakian Formalism}), the induced representation of $G$ on $V_i$ is continuous and unipotent. Hence, the composition $G\rightarrow U_i(k)\rightarrow \GL(V_i)$ is continuous. Since $U_i(k)\hookrightarrow \GL(V_i)$ is a closed topological embedding by Proposition \ref{Prop: topology of a unipotent group}, the map $G\rightarrow U_i(k)$ is continuous.
\end{proof}

\begin{theorem}\label{Thm: equivalence between Malcev completion and Tannakian formalism of unipotent representations}
Let $G$ be a topological group, and let $k$ be a Hausdorff topological field of characteristic zero. Then the continuous homomorphism $G\rightarrow U_{G,k}(k)$ constructed from the neutral Tannakian category $\mathbf{URep}_{k,\cont}(G)$ exhibits $U_{G,k}$ as the continuous Mal'cev $k$-completion of $G$.
\end{theorem}

\begin{proof}
We verify the universal property in Definition \ref{Fact: definition of continuous Malcev completion}. Let $U$ be a pro-unipotent group over $k$, and let $h:G\rightarrow U(k)$ be a continuous homomorphism. Restriction of representations along $h$ defines a $k$-linear exact tensor functor $h^*:\mathbf{Rep}_k(U)\rightarrow \mathbf{URep}_{k,\cont}(G)$. Using the Tannakian equivalence (Theorem \ref{Fact: Tannakian Formalism}), we identify $\mathbf{URep}_{k,\cont}(G)$ with $\mathbf{Rep}_k(U_{G,k})$ and the functor $h^*$ is induced by a  unique homomorphism $\rho_{U_{G,k}}:U_{G,k}\rightarrow U$ of affine group schemes. By Lemma \ref{Lemma: the construction of continuous Malcev completion by Tannakian formalism}, $h$ factors as $G\rightarrow U_{G,k}(k)\xrightarrow{\rho_{U_{G,k}}(k)} U(k)$.
\end{proof}

\section{Hopf-Algebraic Construction}

In this section, we give a Hopf-algebraic construction for the continuous Mal'cev $\QQ_p$-completion of topologically finitely generated pro-$p$ groups.

\subsection{$\QQ_p$-Iwasawa Algebra of a Pro-$p$ Group}\,

In this part, we construct the $\QQ_p$-Iwasawa Algebra $\QQ_p[[G]]^{\wedge}_I$ for a topologically finitely generated pro-$p$ group (Proposition \ref{Prop:I-adic group algebra over Q_p}).

\begin{definition}[\cite{Nikolov-Segal-profinite-groups}*{p.~2}\cite{Ribes-Zallesskii-profinite-groups}*{p.~120}]\label{Def: strongly complete profinite groups}
A profinite group $G$ is \textbf{strongly complete} if one of the following equivalent conditions holds.
\begin{enumerate}[leftmargin=0.25in]
    \item Every subgroup of finite index in $G$ is open.
    \item The canonical map from $G$ to the profinite completion of its underlying abstract group is an isomorphism of topological groups.
    \item Every group homomorphism from $G$ to a profinite group is continuous.
\end{enumerate}
\end{definition}

\begin{theorem}[\cite{Nikolov-Segal-profinite-groups}*{Theorem 1.1, Theorem 1.4}]\label{Fact: finitely generated profinite groups are good}
Let $G$ be a topologically finitely generated profinite group. Then 
\begin{enumerate}[leftmargin=0.25in]
    \item $G$ is strongly complete;
    \item every term $\Gamma^r G$ of the lower central series is closed.
\end{enumerate}
\end{theorem}

\begin{definition}\label{Def: Iwasawa algebra}
The \textbf{Iwasawa algebra} of a pro-$p$ group $G$ is $\widehat{\ZZ}_p[[G]]:=\varprojlim_H\widehat{\ZZ}_p[G/H]$, where $H$ ranges over the open normal subgroups of $G$.
\end{definition}

Equipped with the inverse-limit topology, $\widehat{\ZZ}_p[[G]]$ is a compact Hausdorff totally disconnected topological ring. The quotient maps $G\rightarrow G/H$ gives the following canonical embedding.
\begin{proposition}\label{Prop: canonical map from G to Iwasawa algebra}
Let $G$ be a pro-$p$ group. For each open normal subgroup $H\subset G$, the 
canonical map $G/H\rightarrow \widehat{\ZZ}_p[G/H]^{\times}$ is compatible with passage to quotients and therefore induces a continuous group homomorphism $G\rightarrow \widehat{\ZZ}_p[[G]]^{\times}$.   
\end{proposition}

\begin{example}[\cite{Lazard-Iwasawa-algebra},\cite{Serre-Galois-cohomology}*{Section I.1.5, Proposition 6 and Proposition 7}]\label{Example: Magnus algbra}
Let $F_d$ be the free group on generators $\{x_1,...,x_d\}$. Let $\widehat{F}_d$ be the pro-$p$ completion of $F_d$. Let $\widehat{\ZZ}_p\langle\langle t_1,...,t_d\rangle\rangle$ denote the Magnus algebra of noncommutative formal power series in $t_1,...,t_d$. Then the assignment $x_i\rightarrow 1+t_i$ extends uniquely to an isomorphism of $\widehat{\ZZ}_p$-algebras $s:\widehat{\ZZ}_p[[\widehat{F}_d]]\rightarrow \widehat{\ZZ}_p\langle\langle t_1,...,t_d\rangle\rangle$. \qed
\end{example}

Define $\widehat{\ZZ}_p[[G]]\widehat{\otimes}_{\widehat{\ZZ}_p} \widehat{\ZZ}_p[[G]]:=\varprojlim_{\text{open }H\triangleleft G} (\widehat{\ZZ}_p[G/H]\otimes_{\widehat{\ZZ}_p}\widehat{\ZZ}_p[G/H])$. There is then a canonical map $j:\widehat{\ZZ}_p[[G]]\otimes_{\widehat{\ZZ}_p}\widehat{\ZZ}_p[[G]]\rightarrow \widehat{\ZZ}_p[[G]]\widehat{\otimes}_{\widehat{\ZZ}_p}\widehat{\ZZ}_p[[G]]$. For any $x,y\in\widehat{\ZZ}_p[[G]]$, write $x\widehat{\otimes}y:=j(x\otimes y)$.

For open normal subgroups $H_1\subset H_2$, the quotient map $G/H_1\twoheadrightarrow G/H_2$ is compatible with the usual augmentations of the corresponding group rings. Passing to the inverse limit gives a continuous augmentation $\epsilon:\widehat{\ZZ}_p[[G]]\rightarrow \widehat{\ZZ}_p$.
Similarly, the coproducts and antipodes on the finite group rings are compatible with the transition maps and induce maps $\Delta:\widehat{\ZZ}_p[[G]]\rightarrow \widehat{\ZZ}_p[[G]]\widehat{\otimes}_{\widehat{\ZZ}_p} \widehat{\ZZ}_p[[G]]$ and  $S:\widehat{\ZZ}_p[[G]]\rightarrow \widehat{\ZZ}_p[[G]]$. Thus, we obtain the following.

\begin{proposition}\label{Prop: Iwasawa algebra is a topological Hopf algebra}
Let $G$ be a pro-$p$ group. Then $\widehat{\ZZ}_p[[G]]$ is a topological Hopf algebra, and for every $x\in G$, $\Delta(x)=x\widehat{\otimes} x$, $\epsilon(x)=1$ and $S(x)=x^{-1}$.
\end{proposition}

\begin{proposition}
Let $G$ be a pro-$p$ group. Then the canonical map $G\rightarrow \widehat{\ZZ}_p[[G]]^{\times}$ identifies $G$ homeomorphically with the space $\mathcal{G}(\widehat{\ZZ}_p[[G]])$ of grouplike elements.
\end{proposition}

\begin{proof}
Since $G$ is compact and $\widehat{\ZZ}_p[[G]]$ is Hausdorff, it suffices to show that the map $\phi:G\rightarrow \mathcal{G}(\widehat{\ZZ}_p[[G]])$ induced by Proposition \ref{Prop: Iwasawa algebra is a topological Hopf algebra} is bijective. For every $x\in \mathcal{G}(\widehat{\ZZ}_p[[G]])$ and every open normal subgroup $H\subset G$, the image $x_H\in \widehat{\ZZ}_p[G/H]$ is group-like. Since the group-like elements of $\widehat{\ZZ}_p[G/H]$ are precisely $G/H$, the compatible family $(x_H)_H$ therefore defines an element of $\varprojlim_HG/H=G$. Thus, $\phi$ is bijective.
\end{proof}

Let $I=\Ker(\epsilon:\widehat{\ZZ}_p[[G]]\rightarrow \widehat{\ZZ}_p)$ be the augmentation ideal of $\widehat{\ZZ}_p[[G]]$, and, for each open normal subgroup $H\subset G$, let $I_{G/H}=\Ker(\epsilon_H:\widehat{\ZZ}_p[G/H]\rightarrow \widehat{\ZZ}_p)$. Write $\overline{I^n}$ for the closure of the algebraic ideal $I^n$ in $\widehat{\ZZ}_p[[G]]$.

\begin{proposition}\label{Prop: computation of power of augmentation ideal}
As closed ideals,
$\overline{I^n}=\varprojlim_{\text{open } H\triangleleft G} (I_{G/H})^n$.
\end{proposition}

\begin{proof}
Let $\pi_H:\widehat{\ZZ}_p[[G]]\twoheadrightarrow \widehat{\ZZ}_p[G/H]$ be the canonical projection. Since $\pi_H(I)=I_{G/H}$, we have $\pi_{H}(I^n)=(I_{G/H})^n$. Thus, $I^n\subset \varprojlim_H (I_{G/H})^n$. The right-hand side is closed, so it also contains $\overline{I^n}$. Conversely, let $x=(x_H)_H$ be an element of $\varprojlim_H (I_{G/H})^n$.  The subsets $U_{H,m}=\pi^{-1}_H(p^m\cdot \widehat{\ZZ}_p[G/H])=\Ker(\pi_H)+p^m\cdot \widehat{\ZZ}_p[[G]]$ form a fundamental system of neighborhoods of $0$ in $\widehat{\ZZ}_p[[G]]$. It suffices to show that every neighborhood $x+U_{H,n}$ meets $I^n$. Because $x_H\in \pi_H(I^n)=(I_{G/H})^n$, choose $y_{H,m}\in I^n$ with $\pi_H(y_{H,m})=x_H$. Then $y_{H,m}-x\in \Ker(\pi_H)\subset U_{H,m}$, so $x\in \overline{I^n}$.
\end{proof}

\begin{lemma}\label{Lem: generators of augmentation ideal}
The augmentation ideal $I$ is the closed two-sided ideal generated by $\{g-1:g\in G\}$.
\end{lemma}

\begin{proof}
Let $J$ denote the closed two-sided ideal generated by $\{g-1:g\in G\}$. Since $I$ is closed in $\widehat{\ZZ}_p[[G]]$ and $g-1\in I$ for any $g\in G$, $J\subset I$. To prove the reverse inclusion, it suffices to approximate every $x\in I$ modulo the neighborhoods $U_{H,m}$ introduced in Proposition \ref{Prop: computation of power of augmentation ideal}, it remains to prove that for any $x\in I$, any open normal subgroup $H$ of $G$. Since $\pi_H(x)\in I_{G/H}$, and $I_{G/H}$ is generated as a two-sided ideal by $\{\pi_H(g)-1:g\in G\}$,  we may write $\pi_H(x)$ as a finite sum $\sum_i \pi_H(\alpha_i) (\pi_H(g_i)-1)\pi_H(\beta_i)$ for suitable $\alpha_i,\beta_i\in \widehat{\ZZ}_p[[G]]$ and $g_i\in G$. Set $y:=\sum_i \alpha_i (g_i-1)\beta_i\in J$. Then $y-x\in \Ker(\pi_H)\subset U_{H,m}$. Hence $x\in J$, proving $I=J$.
\end{proof}

\begin{proposition}\label{Prop: the quotient spaces of filtrations of Iwasawa algebra has finite rank}
If $G$ be a topologically finitely generated pro-$p$ group, then
$\widehat{\ZZ}_p[[G]]/\overline{I^n}$ is a finitely generated $\widehat{\ZZ}_p$-module for every $n$. 
\end{proposition}

\begin{proof}
It suffices to prove that each quotient $\overline{I^n}/\overline{I^{n+1}}$ is a finitely generated $\widehat{\ZZ}_p$-module.
Since $G$ is topologically finitely generated, there is a continuous surjective homomorphism $\widehat{F}_d\twoheadrightarrow G$, where $\widehat{F}_d$ is a finitely generated free pro-$p$ group. By Example \ref{Example: Magnus algbra}, the induced homomorphism $\rho:\widehat{\ZZ}_p\langle\langle t_1,...,t_d\rangle\rangle\cong \widehat{\ZZ}_p[[\widehat{F}_d]]\twoheadrightarrow \widehat{\ZZ}_p[[G]]$ is a  continuous surjection. Let $J$ be the augmentation ideal of $\widehat{\ZZ}_p\langle\langle t_1,...,t_d\rangle\rangle$. Then $\rho(J)=I$, and continuity together with compactness gives $\rho(J^n)=\overline{I^n}$.
The quotient $\overline{J^n}/\overline{J^{n+1}}$ is free of rank $d^n$, with basis given by monomials of length $n$ in $t_1,...,t_d$. $\overline{I^n}/\overline{I^{n+1}}$ is therefore finitely generated.
\end{proof}

Let $F$ be a field. Let $\mathbf{Vect}_F$ denote the tensor category of finite-dimensional $F$-vector spaces, and let $\pro-\mathbf{Vect}_F$ denote its pro-category.

\begin{definition}[\cite{Betts-thesis}*{Definition 2.1.21}] 
A \textbf{pro-finite-dimensional $F$-vector space} $V$ is an object $\{V_i\}$ in $\pro-\mathbf{Vect}_F$. We formally write $V=\varprojlim_i V_i$. For pro-finite-dimensional vector spaces $V=\varprojlim_i V_i$ and $W=\varprojlim_j W_j$, define their \textbf{completed tensor product}  $V\widehat{\otimes }W$ by $\varprojlim_{i,j} (V_i\otimes W_j)$. This construction makes $\pro-\mathbf{Vect}_F$ a tensor category.
\end{definition}

\begin{definition}[\cite{Betts-thesis}*{Definition 2.1.22}]
The \textbf{decompleted dual} of $V=\varprojlim_i V_i$ is the $F$-vector space $V^{\vee}:=\varinjlim_i V^*_i$, where $V^*_i=\homo_F(V_i,F)$.
\end{definition}

\begin{lemma}[\cite{Betts-thesis}*{Lemma 2.1.23}]\label{Fact: categorical equivalence between pro-finite-dimensional spaces and vector spaces}
The decompleted-dual functor induces a tensor categorical equivalence from $(\pro-\mathbf{Vect}_F)^{op}$ to all $F$-vector spaces. In particular, $\pro-\mathbf{Vect}_F$ is an $F$-linear rigid tensor category.
\end{lemma}

Using this completed tensor product, one defines  \textbf{pro-finite-dimensional Hopf algebras over $F$}; see \cite{Betts-thesis}*{Definition 2.1.24}.

\begin{definition}
Let $G$ be a topologically finitely generated pro-$p$ group, and let $I$ denote the augmentation ideal of the Iwasawa algebra $\widehat{\ZZ}_p[[G]]$. We call $\QQ_p[[G]]^{\wedge}_I:=\varprojlim_n \QQ_p\otimes_{\widehat{\ZZ}_p} (\widehat{\ZZ}_p[[G]]/\overline{I^n})$ \textbf{$\QQ_p$-Iwasawa algebra} of $G$.
\end{definition}

\begin{proposition}\label{Prop:I-adic group algebra over Q_p}
The $\QQ_p$-Iwasawa algebra $\QQ_p[[G]]^{\wedge}_I$ is a pro-finite-dimensional cocommutative Hopf algebra over $\QQ_p$. Its decompleted dual $(\QQ_p[[G]]^{\wedge}_I)^{\vee}$ is a commutative Hopf algebra. Consequently, $\Spec((\QQ_p[[G]]^{\wedge}_I)^{\vee})$ is an affine group scheme.
\end{proposition}

\begin{proof}
Once the first sentence is established, the second and third sentences follow immediately from the tensor equivalence in Lemma \ref{Fact: categorical equivalence between pro-finite-dimensional spaces and vector spaces}. By Proposition \ref{Prop: the quotient spaces of filtrations of Iwasawa algebra has finite rank}, each quotient $\QQ_p\otimes_{\widehat{\ZZ}_p} (\widehat{\ZZ}_p[[G]]/\overline{I^n})$ is finite-dimensional over $\QQ_p$. Hence, $\QQ_p[[G]]^{\wedge}_I$ is pro-finite-dimensional. Since multiplication on $\widehat{\ZZ}_p[[G]]$ respects the $I$-adic filtration, it induces a  continuous multiplication $\nu:\QQ_p[[G]]^{\wedge}_I\widehat{\otimes}\QQ_p[[G]]^{\wedge}_I\rightarrow \QQ_p[[G]]^{\wedge}_I$. Together with the induced unit and augmentation, this makes $\QQ_p[[G]]^{\wedge}_I$ a pro-finite-dimensional augmented associative $\QQ_p$-algebra.     

Since $\widehat{\ZZ}_p[[G]]$ is torsion-free, $\widehat{\ZZ}_p[[G]]$ is flat over $\widehat{\ZZ}_p$. By Lemma \ref{Lem: tensor product of quotient modules}, $(\widehat{\ZZ}_p[[G]]/\overline{I^m})\otimes_{\widehat{\ZZ}_p} (\widehat{\ZZ}_p[[G]]/\overline{I^n}) \cong (\widehat{\ZZ}_p[[G]]\otimes_{\widehat{\ZZ}_p} \widehat{\ZZ}_p[[G]])/(\overline{I^m}\otimes_{\widehat{\ZZ}_p}  \widehat{\ZZ}_p[[G]]+ \widehat{\ZZ}_p[[G]]\otimes_{\widehat{\ZZ}_p} \overline{I^n})$, which is a finitely generated $\widehat{\ZZ}_p$-module. It follows that, for every $r$, $(\widehat{\ZZ}_p[[G]]\otimes_{\widehat{\ZZ}_p} \widehat{\ZZ}_p[[G]])/(\sum_{s+t=r} \overline{I^s}\otimes_{\widehat{\ZZ}_p}\overline{I^t})$ is a finitely generated $\widehat{\ZZ}_p$-module.

Because $\Delta(\overline{I^n})$ is contained in the $n$-th tensor-product filtration, $\Delta$ induces compatible maps $\Delta_n: \widehat{\ZZ}_p[[G]]/\overline{I^n}\rightarrow (\widehat{\ZZ}_p[[G]]\otimes_{\widehat{\ZZ}_p} \widehat{\ZZ}_p[[G]])/(\sum_{s+t=n} \overline{I^s}\otimes_{\widehat{\ZZ}_p}\overline{I^t})$. By Lemma \ref{Lem: changing pro system}, these give rise to a coproduct $\Delta$ on $\QQ_p[[G]]^{\wedge}_I$. Since the antipode on $\widehat{\ZZ}_p[[G]]$ preserves the closed $I$-adic filtration $\{\overline{I^n}\}$, it induces an antipode $S$ on $\QQ_p[[G]]^{\wedge}_I$. The Hopf-algebra axioms follow from those on $\widehat{\ZZ}_p[[G]]$.
\end{proof}

The following two lemmas are used in the proof of Proposition \ref{Prop:I-adic group algebra over Q_p}.

\begin{lemma}\label{Lem: tensor product of quotient modules}
Let $R$ be a commutative ring. Let $V$ be an $R$-module. Let $A,B\subset V$ be two $R$-submodules. Then there is a natural isomorphism $(V\otimes_R V)/\Im(A\otimes_R V\oplus V\otimes_R B)\rightarrow (V/A) \otimes_R (V/B)$.
\end{lemma}

\begin{proof}
Consider the short exact sequence $0\rightarrow A\rightarrow V\rightarrow V/A\rightarrow 0$. Right exactness of the tensor product gives $A\otimes_R V\rightarrow V\otimes_R V\rightarrow (V/A)\otimes_R V\rightarrow 0$. Then $(V/A)\otimes_R V\cong \frac{V\otimes_R V}{\Im(A\otimes_R V)}$. Next, tensor the exact sequence $0\rightarrow B\rightarrow V\rightarrow V/B\rightarrow 0$ with $V/A$. This gives $(V/A)\otimes_R B\rightarrow (V/A)\otimes_R V\rightarrow (V/A)\otimes_R (V/B)\rightarrow 0$. Then $(V/A)\otimes_R (V/B)\cong \frac{(V/A)\otimes_R V}{\Im((V/A)\otimes_R B)}\cong \frac{V\otimes_R V}{\Im(A\otimes_R V\oplus V\otimes_R B)}$.
\end{proof}

\begin{lemma}\label{Lem: changing pro system}
Let $R$ be a commutative ring. Let $V$ be a flat $R$-module, and let $I_0=V\supset I_1\supset I_2\supset ...$ be a descending filtration by $R$-submodules. Then there is a canonical isomorphism of the pro $R$-modules $\{(V/I_m)\otimes_{R}(V/I_n)\}_{m,n}$ and $\{(V\otimes_{R} V)/(\sum_{s+t=r}I_s\otimes_{R} I_t)\}_{r}$.
\end{lemma}

\begin{proof}
Since $V$ is flat, the map $A\otimes V\rightarrow V\otimes V$ is injective for every submodule $A\subset V$. Lemma \ref{Lem: tensor product of quotient modules} gives a natural isomorphism $(V/A) \otimes (V/B)\cong (V\otimes V)/(A\otimes V+ V\otimes B)$ for every pair of submodules $A,B$. It therefore suffices to show that the following two pro systems of submodules of $V\otimes_R V$ are cofinal to each other: $\{K_{m,n}=I_m\otimes V+V\otimes I_n\}_{m,n}$ and $\{J_r=\sum_{s+t=r}I_s\otimes I_t\}_r$.  

Since $I_r\otimes_R V$ and $V\otimes_R I_r$ occur among the summands defining $J_r$, we have $K_{n,n}\subset J_n$. Conversely, we claim that $J_{m+n}\subset K_{m,n}$. For every pair $s,t$ with $s+t=m+n$, either $s\geq m$ or $t\geq n$. In the first case, $I_s\otimes I_t\subset I_m\otimes V\subset K_{m,n}$; in the second case, $I_s\otimes I_t\subset V\otimes I_n\subset K_{m,n}$. Thus, every summand of $J_{m+n}$ lies in $K_{m,n}$.
\end{proof}

\subsection{Equivalence to Continuous Mal'cev Completion}\,

In this part, we prove that the affine group scheme $\Spec((\QQ_p[[G]]^{\wedge}_I)^{\vee})$ constructed from the dual of the $\QQ_p$-Iwasawa algebra is the continuous Mal'cev $\QQ_p$-completion of a topologically finitely generated pro-$p$ group $G$.

Let $\QQ_p[[G]]:=\QQ_p\otimes_{\widehat{\ZZ}_p} \widehat{\ZZ}_p[[G]]$. Let $\mathbf{Rep}_{\QQ_p,\cont}(G)$, $\mathbf{Rep}_{\QQ_p,\cont}(\widehat{\ZZ}_p[[G]])$ and $\mathbf{Rep}_{\QQ_p,\cont}(\QQ_p[[G]])$ denote the categories of continuous finite-dimensional  $\QQ_p$-representations of $G$, $\widehat{\ZZ}_p[[G]]$ and $\QQ_p[[G]]$, respectively.

\begin{proposition}\label{Prop: categorical equivalences between continuous representations over groups and group algebras}
Restriction along the continuous maps $G\rightarrow \widehat{\ZZ}_p[[G]]\rightarrow \QQ_p[[G]]$ induces tensor categorical equivalences $\mathbf{Rep}_{\QQ_p,\cont}(\QQ_p[[G]])\rightarrow \mathbf{Rep}_{\QQ_p,\cont}(\widehat{\ZZ}_p[[G]])\rightarrow \mathbf{Rep}_{\QQ_p,\cont}(G)$.
\end{proposition}

\begin{proof}
The first functor is an equivalence because a $\widehat{\ZZ}_p[[G]]$-action on a $\QQ_p$-vector space extends uniquely to $\QQ_p[[G]]=\QQ_p\otimes_{\widehat{\ZZ}_p}\widehat{\ZZ}_p[[G]]$. For the second functor, we extend each continuous $\QQ_p$-representation $V$ of $G$ uniquely to a continuous representation of $\widehat{\ZZ}_p[[G]]$. Denote $n=\dim_{\QQ_p}V$.

\textbf{Claim.} The representation $V$ admits a $G$-stable $\widehat{\ZZ}_p$-lattice $L$. 

\textit{Proof of Claim.} Choose a $\widehat{\ZZ}_p$-lattice $L_0\subset V$. Then the rank of $L_0$ is $n$. After choosing a basis of $L_0$, identify $\GL(V)$ with $\GL(n,\QQ_p)$. The stabilizer of $L_0$ is then $\GL(n,\widehat{\ZZ}_p)$. Since $\GL(n,\widehat{\ZZ}_p)$ is open in $\GL(n,\QQ_p)$, its preimage $H\subset G$ is open. Because $G/H$ is finite, the orbit of $L_0$ under $G$ consists of finitely many lattices $L_0,L_1,...,L_r$. Set $L=L_0+...+L_r$. Then $L$ is a $G$-stable $\widehat{\ZZ}_p$-lattice in $V$. This completes the proof of the claim.
\medskip

With respect to a basis of $L$, the representation takes values in $\GL(n,\widehat{\ZZ}_p)$. For $\alpha\geq 1$, let $G_{\alpha}$ be the kernel of the action of $G$ on $L/p^{\alpha}L$. Then $G_{\alpha}\subset G$ is open and normal. Since $\widehat{\ZZ}_p[[G]]\rightarrow \widehat{\ZZ}_p[G/G_{\alpha}]$ is continuous, $\widehat{\ZZ}_p[[G]]\rightarrow \widehat{\ZZ}_p[G/G_{\alpha}]\rightarrow \End(L/p^{\alpha}L)$ is continuous. Passing to the inverse limit gives a continuous homomorphism $\widehat{\ZZ}_p[[G]]\rightarrow \End(L)$. After extending scalars from $\widehat{\ZZ}_p$ to $\QQ_p$, this makes $V$ a continuous $\QQ_p$-representation of $\widehat{\ZZ}_p[[G]]$. This construction is inverse to the restriction functor and is compatible with tensor products, which completes the proof.
\end{proof}

\begin{definition}\label{Def: I-adically continuous representations}
\begin{enumerate}[leftmargin=0.25in]
    \item An \textbf{$I$-adically continuous finite-dimensional $\QQ_p$-representation} of $\widehat{\ZZ}_p[[G]]$ is a finite-dimensional $\QQ_p$-vector space $V$ together with a continuous $\widehat{\ZZ}_p$-algebra  homomorphism $\widehat{\ZZ}_p[[G]]\rightarrow \End_{\QQ_p}(V)$ whose kernel contains $I^n$ for some $n$.
    \item An \textbf{$I$-adically continuous finite-dimensional $\QQ_p$-representation} of $\QQ_p[[G]]^{\wedge}_I$ is a finite-dimensional $\QQ_p$-vector space $V$ together with a morphism of pro-algebras $\{ \QQ_p\otimes_{\widehat{\ZZ}_p} (\widehat{\ZZ}_p[[G]]/\overline{I^n})\}_n\rightarrow \End_{\QQ_p}(V)$.
\end{enumerate}
\end{definition}

Let $\mathbf{Rep}_{\QQ_p}(\Spec((\QQ_p[[G]]^{\wedge}_I)^{\vee}))$ denote the category of finite-dimensional algebraic $\QQ_p$-representations of this affine group scheme $\Spec((\QQ_p[[G]]^{\wedge}_I)^{\vee})$. Let $\mathbf{Rep}_{\QQ_p,\cont}(\QQ_p[[G]]^{\wedge}_I)$ denote the category of $I$-adically continuous finite-dimensional $\QQ_p$-representations of $\QQ_p[[G]]^{\wedge}_I$.

\begin{proposition}\label{Prop: categorical equivalence between representations over affine group scheme and Hopf algebra}
If $G$ is a topologically finitely generated pro-$p$ group, then 
$\mathbf{Rep}_{\QQ_p}(\Spec((\QQ_p[[G]]^{\wedge}_I)^{\vee}))$ and $\mathbf{Rep}_{\QQ_p,\cont}(\QQ_p[[G]]^{\wedge}_I)$ are canonically equivalent as $\QQ_p$-linear tensor categories.   
\end{proposition}

\begin{proof}
Let $V$ be a finite-dimensional representation of the affine group scheme $\Spec((\QQ_p[[G]]^{\wedge}_I)^{\vee})$. Equivalently, $V$ is a right comodule over the commutative Hopf algebra $\QQ_p[[G]]^{\wedge}_I)^{\vee}$. Write the coaction as $\Delta_V:V\rightarrow V\otimes_{\QQ_p}(\QQ_p[[G]]^{\wedge}_I)^{\vee}$. Since $V$ is finite-dimensional, the image of $\Delta_V$ is contained in $V\otimes_{\QQ_p}(\QQ_p\otimes_{\widehat{\ZZ}_p} \widehat{\ZZ}_p[[G]]/\overline{I^n})^{\vee}$ for some $n$. Define $\nu^n_V:(\QQ_p\otimes_{\widehat{\ZZ}_p} \widehat{\ZZ}_p[[G]]/\overline{I^n})\otimes_{\QQ_p} V\rightarrow V$ by $\nu^n_V(a,v)=(1_V\otimes a)(\Delta_V(v))$. The comodule axioms imply that this defines an associative unital action, making $V$ an $I$-adically continuous representation of $\QQ_p[[G]]^{\wedge}_I$.

Conversely, let $W$ be an $I$-adically continuous finite-dimensional representation of $\QQ_p[[G]]^{\wedge}_I$. The action $\QQ_p[[G]]^{\wedge}_I\widehat{\otimes}_{\QQ_p} W\rightarrow W$ factors through a finite quotient. So it is represented by a map $\mu_W:(\QQ_p\otimes_{\widehat{\ZZ}_p} \widehat{\ZZ}_p[[G]]/\overline{I^n})\otimes_{\QQ_p} W\rightarrow W$ for some $n$. Choose a basis $\{h_j\}$ of $\QQ_p\otimes_{\widehat{\ZZ}_p} \widehat{\ZZ}_p[[G]]/\overline{I^n}$, with dual basis $\{h^{\vee}_j\}$, and a basis $\{e_i\}$ of $W$. Write $\mu_W(h_j\otimes e_i)=\sum_{s}\alpha^{sj}_i e_s$ for $\alpha^{sj}_i\in \QQ_p$. Define $\Delta_W: W\rightarrow W\otimes_{\QQ_p}(\QQ_p\otimes_{\widehat{\ZZ}_p} \widehat{\ZZ}_p[[G]]/\overline{I^n})^{\vee}$ by $\Delta_W(e_i)=\sum_{st}\alpha^{st}_ie_i\otimes h^{\vee}_t$. Associativity and unitality of the action are equivalent to the coassociativity and counit axioms for $\Delta_W$. These two constructions are mutually inverse and compatible with tensor products.
\end{proof}

Let $\mathbf{URep}_{\QQ_p,\cont}(G)$ denote the category of finite-dimensional unipotent continuous $\QQ_p$-representations of $G$, and let $\mathbf{Rep}_{\QQ_p,I,\cont}(\widehat{\ZZ}_p[[G]])$ denote the category of finite-dimensional $I$-adically continuous $\QQ_p$-representations of $\widehat{\ZZ}_p[[G]]$.

\begin{proposition}\label{Prop: categorical equivalence between unipotent representations over a pro-p group and representations over group algebra}
If $G$ is a topologically finitely generated pro-$p$ group, then
$\mathbf{URep}_{\QQ_p,\cont}(G)$, $\mathbf{Rep}_{\QQ_p,I,\cont}(\widehat{\ZZ}_p[[G]])$ and $\mathbf{Rep}_{\QQ_p,\cont}(\QQ_p[[G]]^{\wedge}_I)$ are canonically equivalent as $\QQ_p$-linear tensor categories.
\end{proposition}

\begin{proof}
Proposition \ref{Prop: categorical equivalences between continuous representations over groups and group algebras} and Definition \ref{Def: I-adically continuous representations} give a canonical equivalence of tensor categories $\mathbf{Rep}_{\QQ_p,I,\cont}(\widehat{\ZZ}_p[[G]])$ and $\mathbf{Rep}_{\QQ_p,\cont}(\QQ_p[[G]]^{\wedge}_I)$. To establish a tensor equivalence between $\mathbf{URep}_{\QQ_p,\cont}(G)$ and $\mathbf{Rep}_{\QQ_p,I,\cont}(\widehat{\ZZ}_p[[G]])$, it remains to show that, under the equivalence in Proposition \ref{Prop: categorical equivalences between continuous representations over groups and group algebras}, continuous unipotent representations of $G$ correspond precisely to $I$-adically continuous representations of $\widehat{\ZZ}_p[[G]]$. Let $V$ be a unipotent finite-dimensional representation of $G$. Choose a $G$-stable filtration of $V$ with trivial successive quotients. If the filtration has length $N$, then $(g_1-1_V)\cdot ...\cdot (g_N-1_V)=0$ for all
$g_1,...,g_N\in G$. By Lemma \ref{Lem: generators of augmentation ideal}, this is equivalent to requiring that the associated continuous homomorphism $\widehat{\ZZ}_p[[G]]\rightarrow \End_{\QQ_p}(V)$ vanishes on $\overline{I^N}$. Thus the representation is $I$-adically continuous. Conversely, let $W$ be a finite-dimensional continuous $\QQ_p$-representation of $\widehat{\ZZ}_p[[G]]$. If $\overline{I^N}$ acts trivially on $W$, then every product $(g_1-1)\cdot ...\cdot (g_N-1)$ acts trivially. The filtration $W\supset IW\supset ...\supset \overline{I^n}W=0$ has trivial $G$-action on successive quotients. So $W$ is a unipotent representation of $G$.
\end{proof}

\begin{theorem}\label{Thm: equivalence between mal'cev completion and group-like elements of Hopf algebra}
If $G$ is a topologically finitely generated pro-$p$ group, then the affine group scheme $\Spec((\QQ_p[[G]]^{\wedge}_I)^{\vee})$ is canonically isomorphic to the continuous Mal'cev $\QQ_p$-completion of $G$.
\end{theorem}

\begin{proof}
By Proposition \ref{Prop: categorical equivalence between representations over affine group scheme and Hopf algebra} and Proposition \ref{Prop: categorical equivalence between unipotent representations over a pro-p group and representations over group algebra}, $\mathbf{URep}_{\QQ_p,\cont}(G)$ and $\mathbf{Rep}_{\QQ_p}(\Spec((\QQ_p[[G]]^{\wedge}_I)^{\vee}))$ are canonically equivalent as tensor categories. The Tannakian formalism (Theorem \ref{Fact: Tannakian Formalism}) and Theorem \ref{Thm: equivalence between Malcev completion and Tannakian formalism of unipotent representations} therefore identify $\Spec((\QQ_p[[G]]^{\wedge}_I)^{\vee})$ with the continuous Mal'cev $\QQ_p$-completion of $G$.
\end{proof}

In \cite{Quillen-rational-homotopy}, Quillen identifies the group-like elements of the $I$-adically complete Hopf algebra $\QQ[G]^{\wedge}_I$ of a finitely generated group $G$ with the $\QQ$-points of the Mal'cev $\QQ$-completion of $G$. We then show an analogous result for a topologically finitely generated pro-$p$ group.

Recall that, for a Hopf algebra $H$ over a field $F$, its grouplike elements are coalgebra morphisms $F\rightarrow H$.

\begin{definition}
Define the subset $\mathcal{G}(\QQ_p[[G]]^{\wedge}_I)\subset \QQ_p[[G]]^{\wedge}_I$ of \textbf{grouplike elements} to be $\homo_{\mathrm{coalg}}(\QQ_p, \QQ_p[[G]]^{\wedge}_I)$, where morphisms are taken in the category of pro-finite-dimensional coalgebras.
\end{definition}

\begin{proposition}
The set $\mathcal{G}(\QQ_p[[G]]^{\wedge}_I)$ is naturally a group.    
\end{proposition}

\begin{proof}
Let $x,y\in \mathcal{G}(\QQ_p[[G]]^{\wedge}_I)$. Define their product $x\cdot y$ by $\QQ_p\cong \QQ_p\otimes_{\QQ_p} \QQ_p\xrightarrow{x\widehat{\otimes} y} \QQ_p[[G]]^{\wedge}_I\widehat{\otimes}_{\QQ_p} \QQ_p[[G]]^{\wedge}_I\xrightarrow{\mu} \QQ_p[[G]]^{\wedge}_I$. The bialgebra compatibility implies that $x\cdot y$ is again group-like, and associativity follows from the associativity of $\mu$. The unit $e\in \mathcal{G}(\QQ_p[[G]]^{\wedge}_I)$ is given by the multiplicative unit $\QQ_p\rightarrow  \QQ_p[[G]]^{\wedge}_I$. Define the inverse of $x$ by $x^{-1}:=S\circ x$, where $S$ is the antipode. The antipode axioms show that $x^{-1}$ is group-like and that $x\cdot x^{-1}=x^{-1}\cdot x=e$.
\end{proof}

\begin{theorem}\label{Thm: group-like elements of the Iwasawa algebra is the Malcev completion}
If $G$ is a topologically finitely generated pro-$p$ group, then $\mathcal{G}(\QQ_p[[G]]^{\wedge}_I)$ is canonically isomorphic to the group of $\QQ_p$-points of the continuous Mal'cev $\QQ_p$-completion of $G$.
\end{theorem}

\begin{proof}
By Theorem \ref{Thm: equivalence between mal'cev completion and group-like elements of Hopf algebra}, it suffices to show that $\mathcal{G}(\QQ_p[[G]]^{\wedge}_I)$ is canonically isomorphic to the group of $\QQ_p$-points of $\Spec((\QQ_p[[G]]^{\wedge}_I)^{\vee})$.
By the decompleted-dual equivalence of Lemma \ref{Fact: categorical equivalence between pro-finite-dimensional spaces and vector spaces}, $\mathcal{G}(\QQ_p[[G]]^{\wedge}_I)$ is naturally bijective to the set $\homo_{\QQ_p\mathrm{-alg}}((\QQ_p[[G]]^{\wedge}_I)^{\vee},\QQ_p)$. The latter is, by definition, the set of $\QQ_p$-points of the affine scheme $\Spec((\QQ_p[[G]]^{\wedge}_I)^{\vee})$. The decompleted-dual functor is compatible with multiplication, so the bijection is an isomorphism of groups. 
\end{proof}

\section{$\QQ_p$-Analogue of Mal'cev's Construction}

In this section, we construct a $\QQ_p$-analogue of Mal'cev's original definition of the Mal'cev completion in \cite{Clement-Majewicz-Zyman-theory-nilpotent-groups}*{Section 4.3}\cite{Malcev-nilpotent-torsion-free-groups} for pro-$p$ groups.

\subsection{Nilpotent Pro-$p$ Groups and $p$-adic Analytic Manifolds}\,

We briefly recall the definitions of $p$-adic analytic manifolds and $p$-adic analytic groups from \cite{Dixon-duSautoy-Mann-Segal-analytic-pro-p-groups}. A \textbf{$p$-adic analytic manifold} is a topological space locally modeled on open subsets of $\widehat{\ZZ}_p^n$, whose transition maps are $\QQ_p$-analytic. A \textbf{$p$-adic analytic group} $G$ is a topological group equipped with a $p$-adic analytic manifold structure for which multiplication and inversion are analytic. The tangent space $T_eG$ at the unit $e\in G$ carries a canonical $\QQ_p$-Lie algebra structure.

\begin{lemma}
Let $G$ be an affine algebraic group over $\QQ_p$. Then
\begin{enumerate}[leftmargin=0.25in]
    \item the group $G(\QQ_p)$, equipped with the $p$-adic topology, is a $p$-adic analytic group;
    \item there is a canonical isomorphism $\Lie(G)\rightarrow T_e(G(\QQ_p))$ of Lie algebras.
\end{enumerate}
\end{lemma}

\begin{proof}
By \cite{Deligne-tannakian-categories}*{Corollary 2.5}, $G$ has a faithful finite-dimensional $\QQ_p$-representation. Thus $G$ can be realized as a closed algebraic subgroup of $\GL(n,\QQ_p)$, so $G(\QQ_p)$ is a closed subgroup of  $\GL(n,\QQ_p)$ defined by polynomial equations. It follows from \cite{Dixon-duSautoy-Mann-Segal-analytic-pro-p-groups}*{Theorem 9.6} that $G(\QQ_p)$ is a $p$-adic analytic group. Both $\Lie(\GL(n,\QQ_p))$ and $T_e(\GL(n,\QQ_p))$ are canonically identified with $\mathfrak{gl}(n,\QQ_p)$. Restricting this identification to $G$ proves item (2).
\end{proof}

\begin{definition}
Let $G$ be a topologically finitely generated pro-$p$ group. Let $\Gamma^*G$ be the lower central series of $G$. By Theorem \ref{Fact: finitely generated profinite groups are good}, each term $\Gamma^rG$ is closed in $G$.  The \textbf{nilpotency class} of $G$ is either $\infty$ or the smallest integer $c$ such that $\Gamma^{c+1}G=\{e\}$. $G$ is \textbf{nilpotent} if its nilpotency class is a finite number.   
\end{definition}

\begin{proposition}
Every topologically finitely generated nilpotent pro-$p$ group is $p$-adic analytic.
\end{proposition}

\begin{proof}
This follows by induction on the nilpotency class, using \cite{Dixon-duSautoy-Mann-Segal-analytic-pro-p-groups}*{Theorem 9.7} at each extension step.
\end{proof}

\begin{corollary}\label{Cor: Lie algebra of a nilpotent $p$-adic analytic group is nilpotent}\label{Cor: Lie algebra of nilpotent group is nilpotent}
If $G$ is a topologically finitely generated nilpotent pro-$p$ group, then its $p$-adic Lie algebra $\Lie(G)$ is nilpotent.
\end{corollary}

\begin{proof}
Proceeding by induction on the nilpotency class, it suffices to observe that a central extension of $p$-adic analytic groups $1\rightarrow Z\rightarrow G\rightarrow H\rightarrow 1$ induces a central extension of Lie algebras. Since $Z\subset Z(G)$, conjugation by every element of $Z$ is trivial. Differentiating the conjugation action gives $[\Lie(Z),\Lie(G)]=0$.
\end{proof}

\begin{lemma}
Let $G$ be a topologically finitely generated nilpotent pro-$p$ group. Then its set of torsion elements, denoted $\Tor(G)$, is a closed normal subgroup. Moreover, $\Tor(G)$ is a finite $p$-group.
\end{lemma}

\begin{proof}
By \cite{Clement-Majewicz-Zyman-theory-nilpotent-groups}*{Theorem 2.26}, the torsion elements of a nilpotent group form a normal subgroup. It remains to prove finiteness. Once this is established, $\Tor(G)$ is closed because $G$ is Hausdorff. It is then a $p$-group since every finite subgroup of a pro-$p$ group is a $p$-group by \cite{Dixon-duSautoy-Mann-Segal-analytic-pro-p-groups}*{Proposition 1.26 (iii)}.

We prove finiteness by induction on the nilpotency class $c$. For $c=1$, $G$ is a finitely generated $\widehat{\ZZ}_p$-module, so the lemma is clear. Suppose inductively that the conclusion holds for nilpotent pro-$p$ groups of nilpotency class less than $c$. Consider the central extension $1\rightarrow \Gamma^{c}G\rightarrow G\xrightarrow{\pi} G/\Gamma^{c}G\rightarrow 1$. The quotient $G/\Gamma^{c}G$ has nilpotency class at most $c-1$. By induction, $\Tor(G/\Gamma^{c}G)$ is finite. Set $H:=\pi^{-1}(\Tor(G/\Gamma^{c}G))$. Every torsion element of $G$ maps to a torsion element of the quotient, so $\Tor(G)\subset H$. Since $G$ has nilpotency class $c$, the subgroup  $\Gamma^{c}G$ is central. Since $\Gamma^cG$ is a finitely generated $\widehat{\ZZ}_p$-module, choose a free $\widehat{\ZZ}_p$-submodule $A\subset\Gamma^{c}G $ such that $\Gamma^{c}G/A$ is finite. There is an exact sequence $1\rightarrow \Tor(\Gamma^{c}G)=\Gamma^{c}G/A\rightarrow H/A\rightarrow \Tor(G/\Gamma^{c}G)\rightarrow 1$ of groups. Both the kernel and the quotient are finite, so $H/A$ is finite. Since $A$ is torsion-free, $\Tor(H)\bigcap A=\{e\}$. Thus the quotient map $H\rightarrow H/A$ is injective on $\Tor(H)$. Therefore, $\Tor(H)$ is finite. Since $\Tor(G)\subset \Tor(H)$, the group $\Tor(G)$ is also finite.
\end{proof}

Let $Z_0(G)=0\subset Z_1(G)=Z(G)\subset ...$ be the upper central series of a group $G$ (\cite{Clement-Majewicz-Zyman-theory-nilpotent-groups}*{Definition 2.5}). If $G$ is a Hausdorff topological group, then its center $Z(G)$ is closed. An induction using $Z_{n+1}/Z_n(G)=Z(G/Z_n(G))$ shows that every $Z_n(G)$ is closed.

\begin{proposition}\label{Prop: the upper central series of torsion-free nilpotent group is torsion-free}
Let $G$ be a topologically finitely generated nilpotent pro-$p$ group. If $G$ is torsion free, then $Z_{n+1}(G)/Z_{n}(G)$ is torsion free.
\end{proposition}

\begin{proof}
For $n=0$, the proposition follows since $Z_1(G)$ is a subgroup of $G$. Assume inductively that $Z_{n}(G)/Z_{n-1}(G)$ is torsion free. Let $x\in Z_{n+1}(G)$, and suppose that $x^m\in Z_n(G)$ for some $m\geq 1$. By the construction of the upper central series, for every $g\in G$, the commutator $(x,g)\in Z_n(G)$. Since $x^m\in Z_n(G)$, we have $(x^m,g)\in Z_{n-1}(G)$. Since the image of $(x,g)$ is central in $G/Z_{n-1}(G)$, $(x^m,g)=(x,g)^m$ modulo $Z_{n-1}(G)$. Thus $(x,g)^m\in Z_{n-1}(G)$. By the inductive hypothesis, $Z_{n}(G)/Z_{n-1}(G)$ is torsion free, so $(x,g)\in Z_{n-1}(G)$. Since this holds for every $g\in G$, we obtain $x\in Z_{n}(G)$. Thus, $Z_{n+1}(G)/Z_n(G)$ is torsion-free.
\end{proof}

\begin{definition}\label{Def: Malcev basis}
Let $G$ be a topologically finitely generated nilpotent torsion-free pro-$p$ group, and let $Z_*(G)$ denote its upper central series. For each $i$, choose elements $x_{i1},...,x_{ir_i}$ of $Z_{i}(G)$ whose images form a $\widehat{\ZZ}_p$-basis of $Z_{i}(G)/Z_{i-1} (G)$. The resulting ordered collection of these elements is called a \textbf{Mal'cev basis} for $G$.
\end{definition}

\begin{proposition}\label{Prop: each torsion free nilpotent pro-$p$ group is analytic}
Let $G$ be a topologically finitely generated nilpotent pro-$p$ torsion-free group, and let $\{v_1,...,v_r\}$ be a Mal'cev basis. Then the map $\phi_v:\widehat{\ZZ}_p^r\rightarrow G$ given by $(\lambda_1,...,\lambda_r)\rightarrow v_1^{\lambda_1}...v_r^{\lambda_r}$, defined as in \cite{Dixon-duSautoy-Mann-Segal-analytic-pro-p-groups}*{Proposition 1.26}, is an analytic homeomorphism.    
\end{proposition}

\begin{proof}
The Mal'cev-coordinate decomposition shows that $\phi_v$ is bijective.  By \cite{Dixon-duSautoy-Mann-Segal-analytic-pro-p-groups}*{Proposition 1.26}, $\phi_v$ is continuous. Since $\widehat{\ZZ}_p^r$ is compact and $G$ is Hausdorff, $\phi_v$ is a homeomorphism. For each $i$, the map $\phi_i:\widehat{\ZZ}_p\rightarrow G$ given by $\phi_i(\lambda)=v_i^{\lambda}$ is a continuous homomorphism, and is therefore analytic by \cite{Dixon-duSautoy-Mann-Segal-analytic-pro-p-groups}*{Theorem 9.4}. The factorization $\widehat{\ZZ}_p^r\xrightarrow{\prod_i \phi_i}G^r\xrightarrow{\mu} G$ of $\phi_v$ shows that $\phi_v$ is analytic, where $\mu$ is multiplication.
\end{proof}

Thus, every Mal'cev basis $v_1,...,v_r$ determines a global analytic chart $\phi_v$ on $G$. In this chart, the multiplication on $G$ is represented by a formal group law $F_v=(F_{v,1},...,F_{v,r})$ with $F_{v,i}\in\QQ_p[[x_1,...,x_r,y_1,...,y_r]]$ for every $i$ (\cite{Serre-Lie-algebra-Lie-group}*{Part II, Section IV.6}), characterized by $\phi_v\circ F_v(\alpha,\beta)=\phi_v(\alpha)\cdot \phi_v(\beta)$.

\begin{proposition}\label{Prop: the formal group law of nilpotent group is polynomial}
Let $G$ be a topologically finitely generated nilpotent torsion-free pro-$p$ group, and let $\{v_1,...,v_r\}$ be a Mal'cev basis. Then every component $F_{v,i}$ of the associated formal group law $F_v$ is a polynomial.    
\end{proposition}

\begin{proof}
Write $X=(x_1,...,x_r)$ and $Y=(y_1,...,y_r)$, and let $B(X,Y)$ denote the quadratic homogeneous part of $F_v(X,Y)$.
The chart $\phi_v$ identifies $\Lie(G)$ with $\QQ_p^r$, under which the Lie bracket is given by $[X,Y]=B(X,Y)-B(Y,X)$ (\cite{Serre-Lie-algebra-Lie-group}*{Part II, Section V.1}). By \cite{Serre-Lie-algebra-Lie-group}*{Part II, Section V.4}, $F_v$ coincides with the Baker-Campbell-Hausdorff formula associated to this Lie bracket. Since $\Lie(G)$ is nilpotent by Corollary \ref{Cor: Lie algebra of nilpotent group is nilpotent}, its Baker-Campbell-Hausdorff series terminates after finitely many terms. This completes the proof.
\end{proof}

\begin{lemma}\label{Lem: power map of nilpotent group is a polynomial}
Let $G$ be a topologically finitely generated nilpotent torsion-free pro-$p$ group, and let $\{v_1,...,v_r\}$ be a Mal'cev basis. Then there exist polynomials $g_1,...,g_r\in \QQ_p[z,x_1,...,x_r]$ such that $(v_1^{\alpha_1}\cdot ...\cdot v_r^{\alpha_r})^{\gamma}=v_1^{g_1(\gamma,\alpha_1,...,\alpha_r)}
\cdot ...\cdot v_r^{g_r(\gamma,\alpha_1,...,\alpha_r)}$ for all $\gamma,\alpha_1,...,\alpha_r\in \widehat{\ZZ}_p$.
\end{lemma}

\begin{proof}
Let $c$ be the nilpotency class of $G$. When $c=1$, the group $G$ is abelian, and the lemma is immediate. Assume inductively that this lemma holds for all nilpotency classes less than $c$. It suffices to derive the required formula for positive integers $\gamma$. Once the polynomials $g_i$ have been obtained, continuity and the density of $\ZZ_{>0}$ in $\widehat{\ZZ}_p$ extend the result to $\widehat{\ZZ}_p$. Recall the Hall-Petresco words $\tau_2,\tau_3,...$ in variables $y_1,...,y_r$ (\cite{Clement-Majewicz-Zyman-theory-nilpotent-groups}*{Section 4.1.3}). For every integer $k\geq 2$ and all $x_1,...,x_r$ in a group $H$, $x_1^k...x_r^k=(x_1...x_r)^k\cdot \tau_2(x_1,...,x_r)^{\binom{k}{2}}...\cdot \tau_k(x_1,...,x_r)^{\binom{k}{k}}$, with $\tau_j(x_1,...,x_r)\in \Gamma^jH\subset Z_{c+1-j}(H)$.

After ordering the Mal'cev basis appropriately, we may assume that $v_1,...,v_s$ for some $s<r$ form a Mal'cev basis for $Z_{c-1} (G)$. Set $u_k=\tau_k(v_1^{\alpha_1},...,v_r^{\alpha_r})$. Since $u_k\in Z_{c-1} (G)$, it has a unique expression $u_k=v_1^{\alpha_{k,1}}...v_s^{\alpha_{k,s}}$. Proposition \ref{Prop: the formal group law of nilpotent group is polynomial} implies that each coordinate $\alpha_{k,i}$ is a polynomial in $\alpha_1,...,\alpha_r$. Since $\Gamma^{c+1}G=0$, we have $u_k=e$ for every $k>c$. Then $(v_1^{\alpha_1}...v_r^{\alpha_r})^{\gamma}=v_1^{\alpha_1\gamma}...v_r^{\alpha_r\gamma}\cdot u_{c}^{-\binom{\gamma}{c}}\cdot ...\cdot u_{2}^{-\binom{\gamma}{2}}$. For each $k$, write $u_{k}^{-\binom{\gamma}{k}}=v_1^{\beta_{k,1}}...v_s^{\beta_{k,s}}$. By the inductive hypothesis, each $\beta_{k,i}$ is a polynomial in $\gamma,\alpha_1,...,\alpha_r$. Combining these expressions using the polynomial group law of Proposition \ref{Prop: the formal group law of nilpotent group is polynomial} shows that each resulting coordinate $g_i(\gamma,\alpha_1,...,\alpha_r)$ is a polynomial.
\end{proof}

\begin{proposition}\label{Prop: Malcev completion for torsion-free nilpotent group}
Let $G$ be a topologically finitely generated nilpotent torsion-free pro-$p$ group, and let $\{v_1,...,v_r\}$ be a Mal'cev basis. Let $F_v$ be the associated formal group law. Then $F_v$ defines a $p$-adic analytic group structure on $\QQ^r_p$. In particular, the composition of analytic maps $G\xrightarrow{\phi_v^{-1}} (\widehat{\ZZ}^r_p,F_v)\hookrightarrow U_{G,v}:=(\QQ^r_p,F_v)$ is a group homomorphism.
\end{proposition}

\begin{proof}
This follows directly from Proposition \ref{Prop: the formal group law of nilpotent group is polynomial} and Lemma \ref{Lem: power map of nilpotent group is a polynomial}.
\end{proof}

\subsection{Equivalence to Continuous Mal'cev Completion}\,

In this part, we show that the group $U_{G,v}$ constructed in Proposition \ref{Prop: Malcev completion for torsion-free nilpotent group} is the continuous Mal'cev $\QQ_p$-completion of a topologically finitely generated nilpotent pro-$p$ group $G$ (Proposition \ref{Prop: Malcev's original completion is unipotent completion}). Using the lower central series, this construction leads to the continuous Mal'cev $\QQ_p$-completion of the non-nilpotent case (Theorem \ref{Thm: equivalence between malcev completion and formal group law}).

\begin{proposition}\label{Prop: Malcev's original completion is unipotent completion}
Let $G$ be a topologically finitely generated nilpotent pro-$p$ group. Then 
\begin{enumerate}[leftmargin=0.25in]
    \item the canonical Mal'cev map $G\rightarrow (G\widehat{\otimes}\QQ_p)(\QQ_p)$ of Definition \ref{Fact: definition of continuous Malcev completion} is analytic;
    \item Let $v=(v_1,...,v_r)$ be a Mal'cev basis of $G/\Tor(G)$. Then the canonical homomorphism $G\rightarrow U_{G/\Tor(G),v}$ realizes $U_{G/\Tor(G),v}$ as
    the continuous Mal'cev $\QQ_p$-completion of $G$;
    \item the induced morphism $\Lie(G)\rightarrow \Lie(G\widehat{\otimes}\QQ_p)$ is an isomorphism of Lie algebras.
\end{enumerate}
\end{proposition}

\begin{proof}
By \cite{Dixon-duSautoy-Mann-Segal-analytic-pro-p-groups}*{Theorem 9.4}, the canonical Mal'cev map $G\rightarrow (G\widehat{\otimes}\QQ_p)(\QQ_p)$ is analytic. Every continuous homomorphism from $G$ to $U(\QQ_p)$, where $U$ is a unipotent group over $\QQ_p$, kills torsion elements. Therefore, $G$ and $G/\Tor(G)$ have canonically isomorphic continuous Mal'cev $\QQ_p$-completions. Thus, to prove item (2) and item (3), we may assume that $G$ is torsion-free. 

By the correspondence between formal group laws and Lie algebras (\cite{Serre-Lie-algebra-Lie-group}*{p.~112}), the homomorphism $G\rightarrow U_{G,v}$ induces an isomorphism $\Lie(G)\rightarrow \Lie(U_{G,v})$. As a $p$-adic analytic manifold, $U_{G,v}$ is identified with $\QQ_p^r$. Since both the formal group law and the inverse map are polynomial, they define an affine algebraic group $U$ over $\QQ_p$ whose group of $\QQ_p$-points is $U_{G,v}$. Since $\Lie(U)\cong \Lie(G)$ is nilpotent, the algebraic group $U$ is unipotent. Let $U'$ be a unipotent algebraic group over $\QQ_p$, and let $f:G\rightarrow U'(\QQ_p)$ be a  continuous homomorphism. By \cite{Dixon-duSautoy-Mann-Segal-analytic-pro-p-groups}*{Theorem 9.4}, $f$ is analytic and therefore induces a homomorphism $f_*:\Lie(G)\rightarrow \Lie(U')$. Under the equivalence between nilpotent Lie algebras and unipotent algebraic groups (\cite{Betts-thesis}*{Theorem 2.1.4}), $f_*$ determines a unique homomorphism $F:U\rightarrow U'$ of algebraic groups. The Baker-Campbell-Hausdorff formula shows that the composition $G\rightarrow U_{G,v}=U(\QQ_p)\rightarrow U'(\QQ_p)$ agrees with $f$. Hence, $U$ satisfies the universal property of the continuous Mal'cev $\QQ_p$-completion, proving item (2) and item (3).
\end{proof}

\begin{lemma}\label{Lem: Malcev completion preserves exactness of nilpotent groups}
Let $1\rightarrow G_1\rightarrow G_2\rightarrow G_3\rightarrow 1$ be a short exact sequence of topologically finitely generated nilpotent pro-$p$ groups and continuous homomorphisms. Assume that $\QQ_p\subset k$ is a weakly cartesian pair of Hausdorff topological fields. Then the induced sequence $1\rightarrow G_1\widehat{\otimes}k\rightarrow G_2\widehat{\otimes}k\rightarrow G_3\widehat{\otimes}k\rightarrow 1$ is exact.
\end{lemma}

\begin{proof}
The given exact sequence of groups induces a short exact sequence of Lie algebras $0\rightarrow \Lie(G_1)\rightarrow \Lie(G_2)\rightarrow \Lie(G_3)\rightarrow 0$. Proposition \ref{Prop: Malcev's original completion is unipotent completion} therefore proves the assertion for $k=\QQ_p$.
The result for a general $k$ follows by base change from Proposition \ref{Prop: base change of Malcev completion}.
\end{proof}

\begin{proposition}\label{Prop: lower central series of Lie algebra and malcev completion of lower central series}
Let $G$ be a topologically finitely generated pro-$p$ group, and assume that $\QQ_p\subset k$ is a weakly cartesian pair of Hausdorff topological fields. Then, for every $i$,
the Lie algebra $\Lie(G\widehat{\otimes}k)/\Gamma^i\Lie(G\widehat{\otimes}k)$ is naturally isomorphic to $\Lie((G/\Gamma^iG)\widehat{\otimes}k)$.
\end{proposition}

\begin{proof}
By Proposition \ref{Prop: base change of Malcev completion}, it suffices to prove the result when $k=\QQ_p$. The surjection $\Lie(G\widehat{\otimes}\QQ_p)\twoheadrightarrow\Lie(G\widehat{\otimes}\QQ_p)/\Gamma^i\Lie(G\widehat{\otimes}\QQ_p)$ has the universal property that $\Lie(G\widehat{\otimes}\QQ_p)/\Gamma^i\Lie(G\widehat{\otimes}\QQ_p)$ has nilpotency class at most $i-1$ and that, for every Lie algebra $L$ with $\Gamma^iL=0$, every Lie-algebra map $\Lie(G\widehat{\otimes}\QQ_p)\rightarrow L$ uniquely factors through $\Lie(G\widehat{\otimes}\QQ_p)/\Gamma^i\Lie(G\widehat{\otimes}\QQ_p)$. We show that the surjection $\Lie(G\widehat{\otimes}\QQ_p)\twoheadrightarrow\Lie((G/\Gamma^iG)\widehat{\otimes}\QQ_p)$ has the same universal property.

Let $L$ be a Lie algebra with $\Gamma^iL=0$, and let $U$ be a unipotent group whose Lie algebra is $L$. Since $\Gamma^iU(\QQ_p)=0$,  every continuous homomorphism $G\rightarrow U(\QQ_p)$ uniquely factors through $G/\Gamma^iG$. Then  $\homo_{\mathrm{Lie}}(\Lie(G\widehat{\otimes}\QQ_p),L)=\homo(G\widehat{\otimes}\QQ_p,U)=\homo_{\cont}(G,U(\QQ_p))=\homo_{\cont}(G/\Gamma^iG,U(\QQ_p))=\homo((G/\Gamma^iG)\widehat{\otimes}\QQ_p,U)=\homo_{\mathrm{Lie}}(\Lie((G/\Gamma^iG)\widehat{\otimes}\QQ_p),L)$. Thus, every Lie-algebra morphism $\Lie(G\widehat{\otimes}\QQ_p)\rightarrow L$ uniquely factors through $\Lie((G/\Gamma^iG)\widehat{\otimes}\QQ_p)$. Since the nilpotency class of $G/\Gamma^iG$ is at most $i-1$, under the identification $\Lie((G/\Gamma^iG)\widehat{\otimes}\QQ_p)=\Lie(G/\Gamma^iG)$ in Proposition \ref{Prop: Malcev's original completion is unipotent completion}, $\Lie((G/\Gamma^iG)\widehat{\otimes}\QQ_p)$ has nilpotency class at most $i-1$. Therefore, $G\widehat{\otimes}\QQ_p\twoheadrightarrow (G/\Gamma^iG)\widehat{\otimes}\QQ_p$ and $\Lie(G\widehat{\otimes}\QQ_p)\twoheadrightarrow\Lie((G/\Gamma^iG)\widehat{\otimes}\QQ_p)$ have the same universal property. This completes the proof.
\end{proof}

Let $G$ be a topologically finitely generated pro-$p$ group, and let $\Gamma^*G$ be the lower central series.  For $i\geq 1$, set $G_i=G/\Gamma^iG$. For every $i$, choose a Mal'cev basis $v^i$ for $G_i/\Tor(G_i)$.  By Proposition \ref{Prop: Malcev's original completion is unipotent completion} and Lemma \ref{Lem: Malcev completion preserves exactness of nilpotent groups}, the tower $1=G_1\twoheadleftarrow G_2\twoheadleftarrow G_3\twoheadleftarrow ...$ induces a tower $1=U_{G_1/\Tor(G_1),v^1}\twoheadleftarrow U_{G_2/\Tor(G_2),v^2}\twoheadleftarrow U_{G_3/\Tor(G_3),v^3}\twoheadleftarrow ...$ of $p$-adic analytic groups.

\begin{theorem}\label{Thm: equivalence between malcev completion and formal group law}
Let $G$ be a topologically finitely generated pro-$p$ group.
With the notation of the preceding paragraph, the canonical homomorphism $G\rightarrow \varprojlim_i U_{G_i/\Tor(G_i),v^i}$ is the continuous Mal'cev $\QQ_p$-completion of $G$.
\end{theorem}

\begin{proof}
This follows immediately by applying Proposition \ref{Prop: Malcev's original completion is unipotent completion} to $G_i$ and then using the inverse-limit description of Proposition \ref{Prop: Malcev completion is the inverse limit of Malcev completion of lower central series}.
\end{proof}

\appendix

\section{Topological Fields and Topological Vector Spaces}

In this appendix, we review basic properties of topological fields and finite-dimensional topological vector spaces.

A \textbf{topological field} is a field $k$ equipped with a topology such that addition, multiplication,  additive inversion and multiplicative inversion are continuous (\cite{Wedhorn-Adic-spaces}*{Definition 5.13}). 

\begin{proposition}[\cite{Wedhorn-Adic-spaces}*{Remark 5.15}]
$k$ is Hausdorff if and only if the topology on $k$ is nontrivial.    
\end{proposition}

A \textbf{topological vector space} over a topological field $k$ is a vector space $V$ equipped with a topology such that vector addition and scalar multiplication are continuous. We write $k^n$ for the vector space $k^n$ equipped with the product topology.

\begin{lemma}\label{Lemma: continuity of kn}
Let $k$ be a topological field. Then
\begin{enumerate}[leftmargin=0.25in]
   \item every polynomial map $f:k^m\rightarrow k^n$ is continuous;
    \item every linear automorphism of $k^n$ is a homeomorphism.
\end{enumerate}
\end{lemma}

\begin{proof}
Item (1) follows because each coordinate of a polynomial map is obtained from the coordinate projections by finitely many applications of addition and multiplication. For item (2), both the linear automorphism and its inverse are polynomial maps, so both are continuous by item (1).
\end{proof}

\begin{corollary}\label{Cor: subvarieties are closed}
Let $k$ be a Hausdorff topological field. Then every affine subvariety in $k^n$ is closed. In particular, every vector subspace $V$ of $k^n$ is closed and $V$ is linearly homeomorphic to $k^r$, where $r=\dim_{k}(V)$.
\end{corollary}

\begin{proof}
The first assertion follows from Lemma \ref{Lemma: continuity of kn} because an affine variety is the common zero locus of polynomial maps. The second follows by choosing a basis of $V$, extending it to a basis of $k^n$, and applying Lemma \ref{Lemma: continuity of kn}.
\end{proof}

\begin{lemma}\label{Lem: linar maps from K^n is continuous}
Let $V$ be a finite-dimensional topological vector space. Then every linear map $f:k^r\rightarrow V$ is continuous.
\end{lemma}

\begin{proof}
Let $e_1,...,e_r$ be the standard basis of $k^r$. Set $v_i=f(e_i)$. Then $f(\alpha_1,...,\alpha_r)=\alpha_1v_1+...+\alpha_rv_r$ for every $(\alpha_1,...,\alpha_r)\in k^r$. 
Since scalar multiplication on $V$ is continuous, the map $k\rightarrow V$ given by $\alpha\rightarrow \alpha v_i$ is continuous. Since vector addition is also continuous, $f$ is continuous.
\end{proof}

The following is a direct corollary of Lemma \ref{Lem: linar maps from K^n is continuous}.

\begin{corollary}\label{Corollary: any continuous bijection onto a standard space is a homeomophism}
Let $V$ be a finite-dimensional topological vector space. Any continuous linear bijection $f:V\rightarrow k^r$ is a homeomorphism.
\end{corollary}

\begin{definition}\label{Def: Hahn-Banach}
A topological vector space $V$ over $k$ is \textbf{Hahn-Banach} if, for every nonzero $v\in V$, there exists a continuous linear functional $f:V\rightarrow k$ with $f(v)\neq 0$.
\end{definition}

Let $V^{\vee}$ denote the algebraic dual of $V$ and let $\homo_{k,\cont}(V,k)$ denote the vector space of continuous $k$-linear maps $V\rightarrow k$.

\begin{lemma}\label{Lemma: Hahn-Banach is equivalent to all linear functionals are continuous}
A finite-dimensional topological vector space $V$ is Hahn-Banach if and only if the natural inclusion $\homo_{k,\cont}(V,k)\hookrightarrow V^{\vee}$ is an isomorphism.
\end{lemma}

\begin{proof}
The ``if'' part is immediate. We prove the converse.

Let $n\geq 1$ be the dimension of $V$. It suffices to construct $n$ linearly independent continuous linear functionals $f_1,...,f_n:V\rightarrow k$. By Definition \ref{Def: Hahn-Banach}, there exists a nonzero continuous linear map $f_1:V\rightarrow k$. Suppose inductively that we have constructed $f_1,...,f_{m-1}$ with $m\leq n$. Choose a nonzero vector $v\in \bigcap_{i=1}^{m-1}\Ker(f_i)$. Thus, there exists a continuous linear map $f_m:V\rightarrow k$ such that $f_m(v)\neq 0$. If $\sum^{m}_{i=1}\alpha_if_i=0$, the evaluation at $v$ gives $\alpha_m=0$. The inductive hypothesis then implies that $\alpha_1=...=\alpha_{m-1}=0$. This completes the proof.
\end{proof}

\begin{definition}[\cite{Warner-topological-rings}*{Theorem 15.5}]\label{Fact: properties of straight topological field}
Let $k$ be a Hausdorff topological field. The field $k$ is \textbf{straight} if every linear functional on a Hausdorff topological $k$-vector space with closed kernel is continuous.
\end{definition}

\begin{example}[\cite{Warner-topological-rings}*{Theorem 13.8, Theorem 14.12}]\label{Example: Valution fields are straight}
Every valuation field, equipped with its valuation topology, is straight. \qed
\end{example}

\begin{theorem}\label{Thm: definition of weakly cartesian}
Let $V$ be a topological vector space over a straight topological field $k$. Then the following conditions on $V$ are equivalent.
\begin{enumerate}[leftmargin=0.25in]
    \item Every finite-dimensional subspace is linearly homeomorphic to $k^r$ for some $r$.
    \item Every finite-dimensional subspace is closed.
    \item Every finite-dimensional subspace is Hahn-Banach.
\end{enumerate}
\end{theorem}

\begin{proof}
Let $U\subset V$ be a finite-dimensional subspace, and set $n=\dim_k(U)$.

$(1)\Rightarrow (2)$.  Let $x\in \overline{U}$. Let $U'=U+k\cdot x$. Then $U$ and $U'$ are linearly homeomorphic to $k^n$ and $k^s$, respectively, for a suitable integer $s$. Under these identifications, Corollary \ref{Cor: subvarieties are closed} implies that $U$ is closed in $U'$. Then $x\in U$, and thus, $U$ is closed in $V$.

$(2)\Rightarrow (3)$. Let $0\neq x\in U$. Choose a linear functional $f:U\rightarrow k$ such that $f(x)\neq 0$. Because $\Ker(f)$ is finite-dimensional, it is closed in $V$ by assumption. In particular, $\Ker(f)$ is closed in the subspace $U$. Since $k$ is straight, Definition \ref{Fact: properties of straight topological field} implies that $f$ is continuous.

$(3)\Rightarrow (1)$. By Lemma \ref{Lemma: Hahn-Banach is equivalent to all linear functionals are continuous}, there exist $n$ linearly independent continuous linear functionals $f_1,...,f_n:U\rightarrow k$. Define $F=(f_1, ..., f_n):U\rightarrow k^n$. This map is continuous and a linear isomorphism. By Corollary \ref{Corollary: any continuous bijection onto a standard space is a homeomophism}, $F$ is a homeomorphism.
\end{proof}

\begin{definition}
A topological vector space $V$ over a Hausdorff topological field $k$ is \textbf{weakly cartesian} if it satisfies conditions (1) and (2) of Theorem \ref{Thm: definition of weakly cartesian}.
\end{definition}

\begin{proposition}\label{Prop: direct sum of weakly cartesian is weakly cartesian}
Let $k$ be a straight Hausdorff topological field. Then every direct sum of weakly cartesian topological $k$-vector spaces is weakly cartesian. 
\end{proposition}

\begin{proof}
Every finite-dimensional subspace of a direct sum $\bigoplus_{i\in I}V_i$ is contained in a finite direct sum of finite-dimensional subspaces of $V_i$'s. It therefore suffices to consider a direct sum of two finite-dimensional weakly cartesian topological vector spaces $V_1$ and $V_2$. By Theorem \ref{Thm: definition of weakly cartesian}, there is a linear homeomorphism $V_i\cong k^{n_i}$ for each $i=1,2$. Consequently, $V_1\oplus V_2\cong k^{n_1+n_2}$ as topological vector spaces. The proposition follows from Corollary \ref{Cor: subvarieties are closed}.
\end{proof}

\begin{definition}\label{Def: weakly cartesian subfield}
Let $k$ be a closed subfield of a Hausdorff topological field $K$. We call $k\subset K$ a \textbf{weakly cartesian pair} if $K$, regarded as a topological vector space over $k$, is weakly cartesian. In this case, $k$ is called a \textbf{weakly cartesian subfield} of $K$.
\end{definition}

\begin{example}
Let $k$ be a finite extension of $\QQ_p$, equipped with its $p$-adic topology, and let $\CC_p$ denote the completion of the algebraic closure of $\QQ_p$. By Example \ref{Example: Valution fields are straight}, the valuation field $k$ is straight. Every finite-dimensional subspace $V$ over $k$ in $\CC_p$ inherits a norm from $\CC_p$ and is complete. Hence, $V$ is closed in $\CC_p$. Theorem \ref{Thm: definition of weakly cartesian} therefore shows that $k\subset \CC_p$ is a weakly cartesian pair. \qed
\end{example}

\begin{proposition}\label{Prop: base change of vector spaces over weakly cartesian pairs}
Let $k\subset K$ be a weakly cartesian pair of Hausdorff topological fields. Assume that $k$ is straight. Then every weakly cartesian topological $K$-vector space $V$ is also weakly cartesian when regarded as a topological $k$-vector space. In particular, $K^n$ is weakly cartesian as a topological $k$-vector space.
\end{proposition}

\begin{proof}
Let $U=\text{Span}_k(e_1,...,e_n)$ be a finite-dimensional $k$-subspace of $V$. Let $U_K=\text{Span}_K(e_1,...,e_n)$. This is a finite-dimensional $K$-subspace of $V$. Since $V$ is weakly cartesian over $K$, $U_K$ is $K$-linearly homeomorphic to $K^r$ for some $r$. Since $k\subset K$ is a weakly cartesian pair and $k$ is straight, Proposition \ref{Prop: direct sum of weakly cartesian is weakly cartesian} implies that $U_K$ is weakly cartesian over $k$. Therefore, $U$ is linearly homeomorphic to $k^s$ for some $s$. The proposition follows from Theorem \ref{Thm: definition of weakly cartesian}.
\end{proof}

\section{Preliminaries on Nilpotent Groups over Topological Fields}

In this appendix, we recall the definition of pro-unipotent groups over topological fields and describe the natural topology on their groups of rational points.

\begin{definition}[\cite{Betts-thesis}*{Definition-Lemma 2.1.1}]\label{Definition: definition of unipotent groups}
Let $k$ be a field of characteristic zero. An affine algebraic group $U$ over $k$ is \textbf{unipotent} if it satisfies one of the following equivalent conditions.
\begin{enumerate}[leftmargin=0.25in]
    \item Every finite-dimensional representation of $U$ is unipotent.
    \item Every nonzero representation of $U$ has a non-zero fixed vector.
    \item  For some $m$, $U$ admits a closed embedding into the group $\mathbb{U}_{m,k}$ of upper-unitriangular matrices.
    \item $U$ is an iterated (central) extension of the additive group $\mathbb{G}_a$.
\end{enumerate}
An affine group scheme $U$ over $k$ is \textbf{pro-unipotent} if all of its finite-dimensional quotient group schemes are unipotent.
\end{definition}

\begin{proposition}[\cite{Milne-algebraic-groups}*{Proposition 14.32}]\label{Fact: scheme isomorphism between nilpotent Lie algebras and unipotent groups}
Let $k$ be a field of characteristic zero, and let $U$ be a unipotent algebraic group over $k$. Then the exponential map $\exp:\Lie(U)\rightarrow U(k)$ induces an isomorphism of $k$-schemes, and a bijection on $k$-rational points. If $U$ is commutative, this isomorphism is an isomorphism of algebraic groups.
\end{proposition}

\begin{proposition}[\cite{Betts-thesis}*{Lemma 2.1.3}]\label{Fact: closedness of pro-unipotentness for algebraic group schemes}
Let $k$ be a field of characteristic zero. The full subcategory of pro-unipotent groups over $k$ is closed under subgroups, quotients, extensions, and small limits in the category of affine group schemes over $k$. Moreover, this subcategory is categorically equivalent to the pro-category of unipotent groups over $k$.
\end{proposition}

For the remainder of this appendix, we use $k$ to denote a Hausdorff topological field of characteristic zero. We now define the topology on $U(k)$ for every pro-unipotent group $U$ over $k$.

Let $U$ be a unipotent algebraic group over $k$. Choose a linear isomorphism $\Lie(U)\cong k^r$ and equip $\Lie(U)$ with a topology from $k^r$. By Lemma \ref{Lemma: continuity of kn}, the resulting topology is independent of the chosen linear isomorphism $\Lie(U)\cong k^r$. We transport this topology to $U(k)$ through the exponential bijection $\exp:\Lie(U)\rightarrow U(k)$ of Proposition \ref{Fact: scheme isomorphism between nilpotent Lie algebras and unipotent groups}.

\begin{proposition}\label{Prop: topology of a unipotent group}
Let $k$ be a Hausdorff topological field of characteristic zero, and let $U$ be a unipotent group over $k$. Let $\mathbb{U}_{m,k}$ be the algebraic group of unitriangular matrices. Then 
\begin{enumerate}[leftmargin=0.25in]
    \item the exponential map and the logarithm map are mutually inverse homeomorphisms between $\Lie(U)$ and $U(k)$;
    \item every closed embedding of algebraic groups $i:U\hookrightarrow\mathbb{U}_{m,k}$ induces a topological embedding $U(k)\rightarrow \mathbb{U}_{m,k}(k) $ which maps $U(k)$ homeomorphically onto a closed subspace.
\end{enumerate}
\end{proposition}

\begin{proof}
\textbf{Item (1).} Since $U$ is unipotent, both the exponential map and the logarithm map are given by finite polynomial expressions. Both maps are therefore continuous by Lemma \ref{Lemma: continuity of kn}.

\textbf{Item (2).} Since $i$ is a morphism of affine varieties, the induced map $U(k)\rightarrow \mathbb{U}_{m,k}(k) $ is continuous by Lemma \ref{Lemma: continuity of kn}. Moreover, because $i(U)$ is a closed subvariety of $\mathbb{U}_{m,k}$, Corollary \ref{Cor: subvarieties are closed} implies that $i(U(k))$ is closed in $\mathbb{U}_{m,k}(k)$. The induced map $i_*:\Lie(U)\hookrightarrow \Lie(\mathbb{U}_{m,k})$ is continuous by Lemma \ref{Lemma: continuity of kn}  and, by Corollary \ref{Cor: subvarieties are closed}, is a homeomorphism onto its image. Using the exponential homeomorphisms in item (1), we conclude that $U(k)\rightarrow \mathbb{U}_{m,k}(k) $ is a homeomorphism onto its image.
\end{proof}

Let $U$ be a pro-unipotent group over $k$. By Proposition \ref{Fact: closedness of pro-unipotentness for algebraic group schemes}, $U$ can be expressed as an inverse limit $U=\varprojlim_i U_i$ of unipotent groups . We equip $U(k)$ with the inverse limit topology.

\bibliographystyle{amsalpha}
\bibliography{ref}

\Addresses

\end{document}